\documentclass[12pt]{amsart}
\usepackage{hyperref}
\usepackage[utf8]{inputenc}
\usepackage{amssymb, graphicx}
\usepackage[notref,notcite]{showkeys}
\usepackage{amsfonts}
\usepackage{amsmath}
\usepackage{latexsym}
\usepackage{amscd}
\usepackage{verbatim}
\usepackage{mathrsfs,mathtools}
\usepackage{cite}
\allowdisplaybreaks[4]

\newcommand{\mf}{\mathfrak}
\newcommand{\g}{\mf{g}}
\newcommand{\h}{\mf{h}}

\newcommand{\gl}{\mf{gl}}

\newcommand{\Z}{{\mathbb Z}}
\newcommand{\C}{{\mathbb C}}
\newcommand{\N}{{\mathbb N}}
\newcommand{\Q}{{\mathbb Q}}
\newcommand{\R}{\mathbb{R}}
\newcommand{\fH}{\mathfrak{H}}

\newcommand{\K}{{\mathcal K}}

\newcommand{\supp}{{\operatorname{Supp}}\xspace}

\renewcommand{\phi}{\varphi}

\renewcommand{\leq}{\leqslant}
\renewcommand{\geq}{\geqslant}

\def\sl{\mathfrak{sl}}
\def\gl{\mathfrak{gl}}

\def\l{\lambda}

\def\sl{\mathfrak{sl}}

\newtheorem{theorem}{Theorem}[section]
\newtheorem{proposition}[theorem]{Proposition}
\newtheorem{lemma}[theorem]{Lemma}
\newtheorem{corollary}[theorem]{Corollary}

\theoremstyle{remark}

\numberwithin{equation}{section}

\def\span{\mathrm{span}}
\def\supp{\mathrm{Supp}}

\newcommand\pxi[1]{\frac{\partial}{\partial\xi_{#1}}}

\def\Str{{\rm Str}}

\def\diff{{\rm diff}}

\begin{document}




\title[]{Classification of simple strong Harish-Chandra  modules over the Lie superalgebra of vector fields on $\C^{m|n}$}
\author{Yan-an Cai, Rencai L\"{u},  Yaohui Xue}
\maketitle

\begin{abstract}
  In this paper, we classify simple strong Harish-Chandra  modules over the Lie superalgebra $W_{m,n}$ of vector fields on $\C^{m|n}$. Any such module is the unique simple submodule of some tensor module $F(P,V)$ for a simple weight module $P$ over the Weyl superalgebra $\mathcal K_{m,n}$ and a simple weight module $V$ over the general linear superalgebra $\gl_{m,n}$.
\end{abstract}


\vskip 5pt
\noindent
{\em 2010  Math. Subj. Class.:}
17B10, 17B20, 17B65, 17B66, 17B68




\section{Introduction}



We denote by $\Z, \Z_+, \N, \Q$, $\R$, $\R_+$ and $\C$ the sets of all integers, non-negative integers, positive integers, rational numbers, real numbers, non-negative real numbers and complex numbers, respectively. Linear span is denoted by $ \langle\,\rangle_{\R}, \langle\,\rangle_{\R_+}, \langle\,\rangle_{\Z},  \langle\,\rangle_{\Z_+}$, etc., the subscript indicating the coefficients. All vector spaces and algebras in this paper are over $\C$. Any module over a Lie superalgebra or an associative superalgebra is assumed to be $\Z_2$-graded. Let $0\neq (m,n)\in\Z_+^2$ and let $e_1,\dots,e_{m+n}$ be the standard basis of $\C^{m+n}$. For any $r\in\N$, denote by $\overline{1,r}$ the set $\{1,\ldots,r\}$.

Let $A_{m,n}$ (resp. $\mathcal A_{m,n}$) be the tensor superalgebra of the polynomial algebra $\C[t_1,\dots,t_m]$ (resp. Laurent polynomial algebra $\C[t_1^{\pm 1},\dots,t_m^{\pm 1}]$) in $m$ even variables $t_1,\dots,t_m$ and the exterior algebra $\Lambda(n)$ in $n$ odd variables $\xi_1,\dots,\xi_n$. For convenience, sometimes we may write $\xi_i$ as $t_{m+i}$ for $i\in\overline{1,n}$. Denote by $W_{m,n}$ (resp. $\mathcal W_{m,n}$) the Lie superalgebra of super-derivations of $A_{m,n}$ (resp. $\mathcal A_{m,n}$). Cartan $W$-type Lie superalgebra $W_{m,n}$ was introduced by V. Kac in \cite{K}, which is also the Lie superalgebra of vector fields on $\C^{m|n}$.

Weight modules with finite-dimensional weight spaces are called Harish-Chandra modules. The study of of Harish-Chandra modules over $W_{m,n}$ and $\mathcal W_{m,n}$ is a subject of interest by both mathematicians and physicists during the past decades. Such modules over the Virasoro algebra (which is the universal central extension of $\mathcal W_{1,0}$) were conjectured by V. Kac and classified by O. Mathieu in \cite{Ma}, see also \cite{Su2}. Y. Billig and V. Futorny completed the classification for $\mathcal W_{m,0}$ in \cite{BF1}. Simple Harish-Chandra modules over $W_{0,n}$ were classified in \cite{DMP}. Simple strong Harish-Chandra modules over the $N=2$ Ramond algebra (which is a central extension of $\mathcal W_{1,1}$) were classified in \cite{Liu1}. Such modules over $\mathcal W_{m,n}$ were classified independently in \cite{BFIK} and \cite{XL2}. For more related results, we refer the readers to \cite{ BF2, BL,CLL, E1, E2, LZ2, MZ, Sh, Su1, Su2} and the references therein.

In this paper, we will classify simple strong Harish-Chandra modules over $W_{m,n}$. The classification of Harish-Chandra modules over $W_{1,0}$ were given in \cite{Ma}. Very recently, based on the classification of simple bounded $W_{m,0}$ modules (see \cite{CG} for $m=2$ and \cite{XL1} for general $m$) and the characterization of the support of simple weight $W_{m,0}$ modules (see \cite{PS}),  D. Grantcharov and V. Serganova complete the classification of simple Harish-Chandra modules over $W_{m,0}$ (see \cite{GS}). In \cite{LX}, simple bounded weight $W_{m,n}$ modules were classified.

This paper is arranged as follows. In Section 2, we give some definitions, known results and preliminaries. In Section 3, we prove that any simple strong Harish-Chandra module $W_{m,n}$ modules is parabolic induced from a strong bounded module over a subalgebra  $\hat{\mathfrak{k}}=W_{q,n}\ltimes(\mathfrak{k}\otimes A_{q,n})$. In Section 4, simple bounded modules over $\hat{\mathfrak{k}}$ are classified. Finally, in Section 5, main theorem in this paper is given.

\section{Preliminaries}

A vector space $V$ is called a superspace if $V$ is endowed with a $\Z_2$-gradation $V=V_{\bar 0}\oplus V_{\bar 1}$. For any homogeneous element $v\in V$, let $|v|\in\Z_2$ with $v\in V_{|v|}$. Throughout this paper, $v$ is always assumed to be a homogeneous element whenever we write $|v|$ for a vector $v\in V$.

A module over a Lie superalgebra or an associative superalgebra is simple if it does not have nontrivial $\Z_2$-graded submodules. A module $M$ over a Lie superalgebra or an associative superalgebra $\g$ is called strictly simple if $M$ does not have $\g$-invariant subspaces except $0$ and $M$. Clearly, a strictly simple module must be simple. Denote by $\Pi(M)$ the parity-change of $M$ for a module $M$ over a Lie superalgebra or an associative superalgebra.

For any homomorphism of Lie superalgebras or associative superalgebras $\theta: B\to B'$ and any $B'$ module $M$, $M$ becomes a $B$ module, denoted by $M^\theta$, via $b\cdot m:=\theta(b)m$.

\begin{lemma}\cite[Lemma 2.2]{XL2}\label{density}
Let $B,B'$ be two unital associative superalgebras such that $B'$ has a countable basis, $R=B\otimes B'$.
\begin{enumerate}
  \item Let $M$ be a $B$-module and $M'$ be a strictly simple $B'$-module. Then $M\otimes M'$ is a simple $R$-module if and only if $M$ is simple.
  \item Suppose that $V$ is a simple $R$-module and $V$ contains a strictly simple $B'=\C\otimes B'$-submodule $M'$. Then $V\cong M\otimes M'$ for some simple $B$-module $M$.
\end{enumerate}
\end{lemma}

\subsection{Basic definitions.}

We omit $\otimes$ in $A_{m,n}$ for convenience. For any $\alpha=(\alpha_1,\dots,\alpha_m)\in\Z_+^m$ and $i_1,\dots,i_k\in\overline{1,n}$, write $t^\alpha:=t_1^{\alpha_1}\cdots t_m^{\alpha_m}$ and $\xi_{i_1,\dots,i_k}:=\xi_{i_1}\cdots \xi_{i_k}$. Also, for any subset $I=\{i_1,\dots,i_k\} \subset \overline{1,n}$, write $\underline{I}=(l_1,\dots,l_k)$ if $\{l_1,\dots,l_k\}=\{i_1,\dots,i_k\}$ and $l_1<\dots<l_k$. Denote $\xi_I:=\xi_{l_1,\dots,l_k}$ and set $\xi_\varnothing=1$.

Let $i_1,\dots,i_k$ be a sequence in $\overline{1,n}$. Denote by $\tau(i_1,\dots,i_k)$ the inverse order of the sequence $i_1,\dots,i_k$. Let $I,J\subset \overline{1,n}$ with $I\cap J=\varnothing$ and $\underline{I}=(k_1,\dots,k_p),\underline J=(l_1,\dots,l_q)$. We write $\tau(I,J)=(k_1,\dots,k_p,l_1,\dots,l_q)$. Set $\tau(\varnothing,\varnothing)=0$. Then $\xi_{I\cup J}=(-1)^{\tau(I,J)}\xi_I\xi_J$ for all $I\cap J=\varnothing$.

$W_{m,n}$ has a standard basis
$$\{t^\alpha\xi_I\partial_i\ |\ \alpha\in\Z_+^m;I\subset\overline{1,n};i\in\overline{1,m+n}\},$$
where $\partial_i:=\frac{\partial}{\partial t_i}$.

Define the extended Witt superalgebra $\tilde{W}_{m,n}=W_{m,n}\ltimes A_{m,n}$ by
\begin{equation}
[a,a']=0,[x,a]=-(-1)^{|x||a|}[a,x]=x(a),\forall x\in W_{m,n},a,a'\in A_{m,n}.
\end{equation}

Write $d_i:=t_i\partial_i$ for any $i\in\overline{1,m+n}$.  Let $\h_m=\span \{d_i \,|\, i\in\overline{1,m}\}$ and $\fH_{m,n}=\span \{d_i\,|\,i\in\overline{1,m+n}\}$ be the standard Cartan subalgebra of $W_{m,0}$ and $W_{m,n}$ respectively. Let $\epsilon_1,\ldots,\epsilon_{m}$ be the basis in $\h_m^*$  dual to the basis  $d_1,\ldots,d_m$ in $\h_m$.

For the rest of this paper, we fix $m,n\in\mathbb{Z}_+$ with $(m,n)\neq(0,0)$ and write $A:=A_{m,n}$, $W:=W_{m,n}, \mathcal{K}:=\mathcal{K}_{m,n}, \tilde{W}:=\tilde{W}_{m,n}, \gl:=\gl_{m,n}, \h:=\h_{m}, \fH:=\fH_{m,n}$.

Let $\g$ be any Lie super-subalgebra of $\tilde W$ that contains $\h$ and let $M$ be a $\g$ module. In this paper, $M$ is called a weight module with respect to $\h$ if the action of $\h$ on $M$ is diagonalizable. Namely, $M$ is a weight module with respect to $\h$ if $M=\oplus_{\lambda\in \h^*}M_{\lambda}$, where
$$M_{\lambda}=\{v\in M\ |h v=\lambda(h)v,\forall h\in \h, v\in M\}$$
is called the weight space with weight $\lambda$.
Denote by
$$\supp_\h(M)=\{\lambda\in \h^*\ |\ M_{\lambda}\neq 0\}$$
the $\h$-support set of $M$. Denote by $\mathscr{W}_\h(\g)$ the category of weight $\g$ modules with respect to $\h$, and let $\mathscr{F}_{\h}(\g)$ ($\mathscr{B}_\h(\g)$, resp.) be the full subcategory consisting of modules whose weight spaces are finite dimensional (have uniformly bounded dimensions, resp.). Similarly, one can define $\mathscr{W}_\fH$, $\fH$-support, $\mathscr{F}_\fH$ and $\mathscr{B}_\fH$. Modules in $\mathscr{F}_\h(\g)$ and $\mathscr{B}_\h(\g)$ are called strong Harish-Chandra modules and strong bounded modules, if $\fH\subseteq\g$. It is easy to see that simple modules in $\mathscr{F}_\h(\g)$ are contained in $\mathscr{F}_\fH(\g)$. In the rest of this paper, weight modules means modules in $\mathscr{W}_\h$ and support means $\h$-support.

\subsection{Parabolic induction.}
Let us recall some results on parabolic induction. Let  $\Delta(\g)$ be the root set of $\g$ with respect to $\h$, that is $\Delta(\g)\sqcup\{0\}$ is the support of the adjoint module $\g$.

Recall that a subdivision $\Delta(\g)=\Delta(\g)^+\sqcup\Delta(\g)^-$ is called a \emph{triangular decomposition} if and only if $\langle-\Delta(\g)^+\cup\Delta(\g)^-\rangle_{\R_+}\cap\langle-\Delta(\g)^-\cup\Delta(\g)^+\rangle_{\R_+}=\{0\}$. And a \emph{parabolic decomposition} (or \emph{generalized triangular decomposition}) of $\Delta(\g)$ is a subdivision $\Delta(\g)=\Delta(\g)^+\sqcup\Delta(\g)^0\sqcup\Delta(\g)^-$ such that
$p(\Delta(\g)^+)\cap p(\Delta(\g)^-)=\emptyset, 0\notin p(\Delta(\g)^\pm)$, and $p(\Delta(\g)\setminus\Delta(\g)^0)=p(\Delta(\g)^+)\sqcup p(\Delta(\g)^-)$ is a triangular decomposition of $p(\Delta(\g)\setminus\Delta(\g)^0),$ where $p$ is the natural projection $\langle\Delta(\g)\rangle_\R\to\langle\Delta(\g)\rangle_\R/\langle\Delta(\g)^0\rangle_\R$.

If $\Delta(\g)=\Delta(\g)^+\sqcup\Delta(\g)^0\sqcup\Delta(\g)^-$ is a parabolic decompostion, set
\[
\g^0:=\h\oplus(\oplus_{\alpha\in\Delta(\g)^0}\g_\alpha), \, \g^\pm:=\oplus_{\alpha\in\Delta(\g)^\pm}\g_\alpha.
\]
A subalgebra $\mathfrak{p}$ of $\g$ is called \emph{parabolic} if $\mathfrak{p}=\g^0\oplus\g^+$ for some parabolic decomposition.

The following lemma is a superverion of Lemma 2 in \cite{PS}.
\begin{lemma}\label{unique-quotient}
Let $\mathfrak{p}$ be a parabolic subalgebra of $\g$ with a given parabolic decomposition $\Delta(\g)=\Delta(\g)^+\sqcup\Delta(\g)^0\sqcup\Delta(\g)^-$. Let $M$ be a simple weight $\g$ module and $\lambda\in\supp(M)$ such that $\lambda+\alpha\notin\supp(M)$ for any $\alpha\in\Delta(\g)^+$. Then $M$ is the unique quotient of the induced $\g$ module $\mathrm{Ind}_{\mathfrak{p}}^{\g}M^\mathfrak{p}$, where $M^{\mathfrak{p}}:=U(\mathfrak{p})M_\lambda$ is a simple $\mathfrak{p}$ module.
\end{lemma}

\subsection{Modules over $\mathfrak{sl}_2$ and $\mathcal{K}_{m,n}$.}
Now let us recall some results on $\mathfrak{sl}_2$ modules and $\mathcal{K}_{m,n}$ modules. The following result on $\mathfrak{sl}_2$ is well known.
\begin{lemma}\label{sl2}
Let $e,h,f$ be the standard basis of $\sl_2$ and $V\in\mathscr{F}_{\mathbb{C}h}(\mathfrak{sl}_2)$ and $\supp(V)\subseteq \lambda+2\Z$.
\begin{enumerate}
\item If both $e$ and $f$ acts injectively on $V$, then $\supp(V)=\lambda+2\Z$ and $\dim V_\lambda=\dim V_{\lambda+2k}, \forall k\in\Z$.
\item If $e$ acts locally nilpotently on $V$, then $\supp(V)$ is bounded from above.
\end{enumerate}
\end{lemma}
\begin{proof}
The first statement is clear. For the second statement, to the contrary, we have $\lambda+2\Z_+\subseteq\supp(V)$. For any $n>0$, there exists $k\geq n$ and $0\neq v^{(n)}\in V_{\lambda+2k}$ such that $ev^{(n)}=0$. The for all sufficiently large $n$, we have $\lambda\in\supp(U(\mathfrak{sl}_2)v^{(n)})$. Therefore $\dim V_\lambda=\infty$, a contradiction.
\end{proof}

The following properties on $\mathcal{K}$ modules are given in \cite{LX}.

\begin{lemma}\label{Ksimple}
\begin{enumerate}
\item Any simple $\mathcal K$ module is strictly simple.
\item Let $P\in\mathscr{W}_\fH(\mathcal{K})$ be simple. Then $P\cong V_1\otimes \cdots \otimes V_m\otimes \C[\xi_1]\otimes\cdots\otimes \C[\xi_n]$, where every $V_i$ is one of the following simple weight $\C[t_i,\frac{\partial}{\partial t_i}]$-modules: $$t_i^{\lambda_1}\C[t_i^{\pm 1}],\C[t_i],\C[t_i^{\pm 1}]/\C[t_i],$$ where $\lambda_i\in \C\setminus \Z$.
\item Any weight $\mathcal{K}$ module  must have a simple submodule.
\end{enumerate}
\end{lemma}

\subsection{Simple $\mathfrak{gl}_{m,n}$ modules.}
Let $\gl_{m,n}=\gl(\C^{m|n})$ be the general linear Lie superalgebra realized as the spaces of all $(m+n)\times(m+n)$ matrices. Denote by $E_{i,j}, i,j\in\overline{1,m+n}$ be the $(i,j)$-th matrix unit.
$\gl$ has a $\Z$-gradation $\gl=\gl^{-1}\oplus\gl^0\oplus\gl^1$, where
\begin{eqnarray*}
&\gl^{-1}=\span\{E_{m+j,i}\ |\ i\in\overline{1,m},j\in\overline{1,n}\},\\
&\gl^1=\span\{E_{i,m+j}\ |\ i\in\overline{1,m},j\in\overline{1,n}\}
\end{eqnarray*}
and $\gl^0$ is the even part of $\gl$. Obviously, this $\Z$-gradation is consistent with the $\Z_2$-gradation of $\gl$.  

For any module $V$ over the Lie algebra $\gl^0$, $V$ could be viewed as modules over the Lie superalgebra $\gl^0$ with $V_{\bar 0}=V$.
Let $V$ be a $\gl^0$-module and extend $V$ trivially to a $\gl^0\oplus\gl^1$-module.
The \emph{Kac module} of $V$ is the induced module $K(V):=\rm {Ind}_{\gl^0\oplus\gl^1}^{\gl}(V)$. It's easy to see that $K(V)$ is isomorphic to $\Lambda(\gl^{-1})\otimes V$ as superspaces.

\begin{lemma}\cite[Theorem 4.1]{CM}\label{L(V)}
For any simple $\gl^0$ module $V$, the module $K(V)$ has a unique maximal submodule. The unique simple top of $K(V)$ is denoted $L(V)$. Any simple $\gl$ module is isomorphic to $L(V)$ for some simple $\gl^0$-module $V$ up to a parity-change.
\end{lemma}

Clearly, $L(V)\in\mathscr{W}_\h(\gl)$ ($\mathscr{W}_\fH(\gl)$, resp.) if and only if $V\in\mathscr{W}_\h(\gl^0)$ ($\mathscr{W}_\fH(\gl)$, resp.).

\subsection{Differential modules and simplicity of tensor modules}\label{Differential}

From \cite{LX}, we know that there is a Lie superalgebra homomorphism $\pi$ from $W$ to the tensor superalgebra $\mathcal{K}\otimes U(\gl)$ given by
\[
\pi(t^\alpha\xi_I\partial_i)=t^\alpha\xi_I\partial\otimes1+\sum\limits_{s=1}^{m+n}(-1)^{|t_i|(|I|-1)}\partial_s(t^\alpha\xi_I)\otimes E_{s,i}, \alpha\in\Z_+^m, I\subseteq\overline{1,n}, i\in\overline{1,m+n}.
\]
The induced homomorphism from $U(W)$ to $\mathcal{K}\otimes U(\gl)$ will be also denoted by $\pi$. Let $P$ be a $\mathcal{K}$ module and $M$ be a $\gl$ module. Then we have the tensor module $F(P,M)=(P\otimes M)^\pi$ over $W$.

Let $P'\in\mathscr{W}_{\fH_{n,m}}(\mathcal{K}_{n,m})$ be with $\lambda\in \supp(P')$, where $\fH_{n,m}$ is defined as $\fH$ in $\mathcal{K}_{m,n}$. $P'$ may be regard as a $\gl$ module via the monomorphism from $\gl$ to $\mathcal{K}_{n,m}$ defined by mapping $E_{ij}$ to $t_i'\frac{\partial}{\partial t_j'},\forall i,j\in\overline{1,m+n}$, where $(t_1',\ldots,t_{m+n}')=(\xi_1,\ldots,\xi_m,t_1,\ldots,t_n)$. Then the $E=\sum_{i=1}^{m+n}E_{ii}$ eigenspace  $P'[\lambda]:=\{v\in P'|E v=\lambda(E)v\}$ of $P'$ is a simple module in $\mathscr{W}_{\fH_{n,m}}(\gl)$, which will be called a fundamental $\gl$ module.

Note that a fundamental $\gl$ module $P'[\lambda]\in\mathscr{F}_\h(\gl)$ if and only if it is finite-dimensional if and only if up to a parity-change $P'\cong A$ or $P'\cong A^{\sigma}$, where $\sigma\in\mathrm{Aut}(\mathcal{K}_{m,n})$ defined by $\sigma(t_i)=\partial_i, \sigma(\partial_i)=(-1)^{|t_i|+1}t_i$.

Let $P\in\mathscr{W}_{\fH}(\mathcal{K}_{m,n}), P'\in\mathscr{W}_{\fH_{n,m}}(\mathcal{K}_{n,m})$ be simple. Let $$\diff=\sum_{i=1}^{m}\partial_i\otimes\xi_i -\sum_{i=1}^n\partial_{m+i}\otimes t_i\in \mathcal{K}_{m,n}\otimes \mathcal{K}_{n,m}.$$ Then it is easy to verify that $[\diff,\pi(x)]=0,\forall x\in W$. Therefore, $\diff$ is an endomorphism of $W$ module $F(P,P')$ with $\diff^2=0$.

\begin{lemma}\cite{XW}\label{quotient}Let $P\in\mathscr{W}_\fH(\mathcal{K}), M\in\mathscr{W}_{\fH_{n,m}}(\gl)$ be simple.
\begin{enumerate}
\item $F(P,M)$ is simple as $W$ module if $M$ is not isomorphic to a fundamental $\gl$ module.
\item Suppose $M=P'[\lambda]$ for simple weight $\mathcal{K}_{n,m}$ module $P'$.
\begin{enumerate}
\item $F(P,M)$ is not simple if $M$ is neither trivial nor isomorphic to ${\rm Str}_{m,n}, \Pi({\rm Str}_{m,n})$;
\item $\diff(F(P,M))$ is a simple $W$ module if  $M\not\cong {\rm Str}_{m,n},\Pi({\rm Str}_{m,n})$.
\item If $M$ is nontrivial, then  $F(P,M)$ has a unique simple submodule $\diff(F(P, P'[\lambda']))$ for some $\lambda'\in \supp(P')$ with $\lambda'(E)=\lambda(E)-1$.
\item If $M$ is trivial, then $F(P,M)$ is not simple if and only if $P\cong A$ or $\Pi(A)$.
\item If $M={\rm Str}_{m,n}$ or $\Pi({\rm Str}_{m,n})$, then $F(P,M)$ is simple if and only if $P=\sum_{s=1}^{m+n}\partial_sP$.
\end{enumerate}
Here, $\mathrm{Str}_{m,n}=\C$ is the $1$-dimensional $\gl$ module with $x\cdot 1=\mathrm{str}(x), \forall x\in\gl$ with $\mathrm{str}(x)$ the supertrace of $x$.
\end{enumerate}
\end{lemma}

\section{Parabolic induction for $W_{m,n}$} \label{subsec-par-ind}
In this section, we will prove the parabolic induction theorem for $W$. The proof follows from the ideas in \cite{PS}, which is a slightly different. For the rest of this paper, we denote by $\Delta:=\Delta(W)$, that is
\begin{equation}\label{rootset}
\Delta= \left\{\sum_{j=1}^m s_j\epsilon_j,  -\epsilon_i+ \sum_{j\neq i} s_j\epsilon_j \mid s_j \in \mathbb Z_{+}, i\in\overline{1,m}\right\}\setminus\{0\}.
\end{equation}

Consider the embedding $\mathbb C^m\to \mathbb CP^m$. The Lie algebra of vector fields on  $\mathbb CP^m$ is isomorphic to $\mathfrak{sl}_{m+1}$  and is a Lie subalgebra of $W_{m,0}$.  In other words we have a canonical embedding of Lie algebras, $\mathfrak{sl}_{m+1}\subset W_{m,0}$ , with basis
    \[
    d_i,\, e_{\epsilon_i-\epsilon_j}=t_i\partial_j,
    e_{\epsilon_i}=-t_i\sum\limits_{j=1}^mt_j\partial_j,\, e_{-\epsilon_i}=\partial_i,\; i\neq j\in\overline{1,m}.
   \]
   Clearly, with respect to $\h$, $\mathfrak{sl}_{m+1}$ has root set $$\Delta'=\{\epsilon_i-\epsilon_j,\epsilon_i,-\epsilon_i| i\ne j\in\overline{1,m}\}.$$
Denote $\Delta''=\{\epsilon_i-\epsilon_j\,|\,i\ne j\in\overline{1,m}\}.$


Denote $\g=\sl_{m+1}, W_m=W_{m,0}$ for short. Let $M\in\mathscr{F}_\h(W)$ be simple and nontrivial, $\lambda\in \supp(M)$, and $n_{\alpha}^{\lambda}:=\{q\in \mathbb{R}\,|\,\lambda+q\alpha\in \supp(M)\}$. Denote

\begin{equation}\begin{split}&\Delta_M^{'+}(\lambda):=\{\alpha\in \Delta'| n_{\alpha}^{\lambda} \,\,\mbox{is bounded only from above}\},\\
&\Delta_M^{'-}(\lambda):=\{\alpha\in \Delta'| n_{\alpha}^{\lambda}\,\, \mbox{is bounded only from below}\},\\
&\Delta_M^{'F}(\lambda):=\{\alpha\in \Delta'| n_{\alpha}^{\lambda}\,\, \mbox{is bounded in both directions}\},\\
& \Delta_M^{'I}(\lambda):=\{\alpha\in \Delta'| n_{\alpha}^{\lambda}\,\, \mbox{is unbounded in both directions}\}.\end{split} \end{equation}
Then
\begin{equation}\label{shadow}\Delta'=\Delta_M^{'+}(\lambda) \bigsqcup\Delta_M^{'-}(\lambda)\bigsqcup\Delta_M^{'F}(\lambda)\bigsqcup \Delta_M^{’I}(\lambda).\end{equation}

For the rest of this section, we fix $M$ and omit the $M$. Let $\Gamma_{\lambda}=\langle\alpha\in \Delta'\,|\,\lambda+\Z_+\alpha\subset \supp(M)\rangle_{\Z_+}$. 

\begin{lemma}\label{shadow-2}
 \begin{enumerate}
 \item If $\alpha\in \Delta'$, and $\lambda,\lambda+k\alpha\in \supp(M)$ for some $k\in \Z_+$, then $\lambda+j\alpha\in \supp(M)$ for any $1\le i\le k$.
\item  $\Gamma_\lambda$ is independent of $\lambda\in \supp(M)$, so we will omit the $\lambda$.
\item  The subdivision in (\ref{shadow}) does not depend on $\lambda$ with
\[\begin{aligned}
\Delta^{'\pm}&=\{\alpha\in\Delta'\,|\,\pm\alpha\notin\Gamma,\mp\alpha\in\Gamma\},\\
\Delta^{'I}&=\{\alpha\in\Delta'\,|\,\alpha,-\alpha\in\Gamma\},\\
\Delta^{'F}&=\{\alpha\in\Delta'\,|\,\alpha,-\alpha\notin\Gamma\}.
\end{aligned}\]
\end{enumerate}
\end{lemma}

\begin{proof}
The first statement follows from representation theory of $\mathfrak{sl}_2$. The third statement follows directly from the second one. So it remains to prove the second statement.

Fix $\alpha\in\Delta'$ and $\lambda\in\supp(M)$. If $n_\alpha^\lambda$ is bounded from above, then $e_\alpha$ acts locally nilpotently on $M_\lambda$. Hence $e_\alpha$ acts locally nilpotently on $M$ if $\alpha\in\Delta''\cup\{-\epsilon_i\,|\,i\in\overline{1,m}\}$, since $M$ is simple and the adjoint action of $e_{\alpha}$ on $W$ is locally nilpotent. Therefore, for $\alpha\neq\epsilon_i, 1\leq i\leq m$, if $n_\alpha^\lambda$ is bounded from above for some $\lambda$, then $n_\alpha^\lambda$ is bounded from above for all $\lambda$ by Lemma \ref{sl2}. So, the type of $n_\alpha^\lambda$ for $\alpha\neq\epsilon_i$ does not depend on $\lambda$. For $\alpha=\epsilon_i$ with $n_\alpha^\lambda$ bounded from above, the type of $n_{-\alpha}^{-\lambda}$ for the restricted dual of $M$ is bounded from above, and the type does not depend on $\lambda$. From (1), we know that $\Gamma_\lambda$ is independent of $\lambda$.
\end{proof}



For any $\lambda\in \supp(M)$, define $K_{\lambda}:=\{\alpha\in \Delta'|\lambda+\alpha\notin \supp(M)\}, \bar{K}_{\lambda}=\Delta'\setminus K_{\lambda}$. We call $\lambda\in \supp(M)$ extremal if $K_{\lambda}$ is maximal,i.e., $K_{\lambda}$ is not a proper subset of $K_{\mu}$ for any $\mu\in \supp(M)$.

\begin{lemma}\cite[Lemma 5]{PS} \label{PS-Lemma5}
\begin{enumerate}
\item  $\{\epsilon_i,-\epsilon_i\}\cap \Gamma\neq\emptyset$ for all $i$.
\item If $\alpha\in \Delta'^F$, then $\alpha\in \Delta''$, and $\supp(M)$ is invariant with respect to the Weyl group refection $r_{\alpha}$. Moreover, if $\alpha\in K_{\lambda}$, then $(\lambda,\alpha)\in \Z_+$.
\item If $\alpha\in \Gamma$ and $\lambda,\lambda-\alpha\in \supp(M)$, the $K_{\lambda}\subset K_{\lambda-\alpha}$. In particular, if $\lambda$ is extremal, then $\lambda-\alpha$ is also extremal with $K_{\lambda}=K_{\lambda-\alpha}$.
\end{enumerate}
\end{lemma}
\begin{proof}
(1) Suppose to the contrary that $\pm\epsilon_i\notin\Gamma$, then from \cite[Lemma 4]{PS} we have $e_{-\epsilon_i}M=0$. Hence from $[e_{-\epsilon_i}, W]=W$ we know that $M$ is a trivial module, a contradiction.

(2) Since $e_\alpha,e_{-\alpha}$ act locally nilpotently on $M$, the $\mathfrak{sl}_2$-subalgebra generated by $e_\alpha,e_{-\alpha}$ acts locally finitely on $M$. Hence, $\supp(M)$ is $r_\alpha$-invariant. And if $\alpha\in K_\lambda$, $\lambda$ is a highest weight of a finite dimensional $\mathfrak{sl}_2$ module. And the statement follows.

(3) If $\delta\in\bar{K}_{\lambda-\alpha}$, then $\lambda-\alpha+\delta\in\supp(M)$ and hence $\lambda+\delta\in\supp(M)$ since $\alpha\in\Gamma$. Therefore, $\delta\in\bar{K}_\lambda$. So, $K_{\lambda}\subset K_{\lambda-\alpha}$.
%
\end{proof}

\begin{lemma}\cite[Lemma 6]{PS}\label{PS-Lemma6}Let $\lambda\in \supp(M)$ be extremal, and $\alpha,\beta,\alpha+\beta\in \Delta'$, then
\begin{enumerate}
\item If $\alpha,\beta\in K_{\lambda}$, then $\alpha+\beta\in K_{\lambda}$;
\item If $\alpha,\beta\in \bar{K}_{\lambda}$, then $\alpha+\beta\in \bar{K}_{\lambda}$.
\end{enumerate} \end{lemma}
\begin{proof} (1) Suppose to the contrary that  $\alpha+\beta\not\in K_{\lambda}$. First, assume that at least one of the two roots $\alpha,\beta $ (say $\alpha$) belongs to $\Delta''$. Then $\mu=\lambda+\alpha+\beta\in \supp(M)$, $\alpha\in K_{\lambda}, -\alpha\in K_{\mu}$. From Lemma \ref{PS-Lemma5},$(\lambda,\alpha)\in \Z_+$ and $(\mu,\alpha)\in \Z_-$. Since $(\mu-\lambda,\alpha)=(\alpha+\beta,\alpha)=1$, it is impossible.

Let now $\alpha,\beta\not\in \Delta''$. Then without loss of generality, we may assume that $\alpha=-\epsilon_i$ and $\beta=\epsilon_j$. Since $\alpha+\beta=\epsilon_j-\epsilon_i\not\in K_{\lambda}$, we have $\mu=\lambda+\epsilon_j-\epsilon_i\in \supp(M)$. Then $-\epsilon_i,\epsilon_j\in K_{\lambda}$, $\epsilon_i,-\epsilon_j\in K_{\mu}$, which from Lemma \ref{PS-Lemma5} is impossible.

(2) Suppose to the contrary that $\alpha+\beta\in K_{\lambda}$. Then $\lambda+\alpha+\beta\not\in \supp(M), \lambda+\alpha,\lambda+\beta\in \supp(M)$.  As in the proof of (a), we consider two cases. First, let $\alpha\in\Delta''$. If $-\alpha\in\Gamma$, then $\beta\in K_{\lambda+\alpha}=K_{\lambda}$, which is a contradiction. On the other hand $\alpha\not\in \Gamma$ since $\alpha\in K_{\lambda+\beta}$. Therefore from Lemma \ref{PS-Lemma5} $(\lambda+\beta,\alpha)\in \Z_+$. Since $(\beta,-\alpha)=1$, we have $(\lambda,\alpha)\ge1$. Furthermore, $\beta\in K_{\lambda+\alpha}\setminus K_{\lambda}$. Since $\lambda$ is extremal, one can find $\delta\in K_{\lambda}\setminus K_{\lambda+\alpha}$. Then  $\lambda+\alpha+\delta\in \supp(M)$ and $\lambda+\delta\not\in \supp(M)$, which implies $(\lambda+\alpha+\delta,\alpha)\in \Z_-$. But $(\delta,\alpha)\ge -1$ and therefore $(\lambda,\alpha)\in -((\alpha,\alpha)+(\delta,\alpha))+\Z_-\le 0$ which contradicts to $(\lambda,\alpha)\ge 1$.

Now let $\alpha,\beta\not\in \Delta''$, for example, $\alpha=-\epsilon_i,\beta=\epsilon_j$. Then $\lambda-\epsilon_i,\lambda+\epsilon_j\in \supp(M)$, $\lambda+\epsilon_j-\epsilon_i\not\in \supp(M)$. Then $-\epsilon_j\not\in \Gamma$, otherwise $-\epsilon_i\in K_{\lambda}=K_{\lambda+\epsilon_j}$, a contradiction. On the other hand, $\epsilon_j\in K_{\lambda-\epsilon_i}$ implies $\epsilon_j\not\in \Gamma$, a contradiction to Lemma \ref{PS-Lemma5}(a).\end{proof}

\begin{corollary}\cite[Corollary 1]{PS}\label{PS-Corollary1}
Let $\lambda\in\supp(M)$ be extremal. Then $\prescript{}{\lambda}\Delta^{'+}\sqcup\prescript{}{\lambda}\Delta^{'0}\sqcup\prescript{}{\lambda}\Delta^{'-}$ is a parabolic decomposition of $\Delta'$, where
\[\begin{aligned}
\prescript{}{\lambda}\Delta^{'+}&:=K_\lambda\setminus(K_\lambda\cap-K_\lambda),\\
\prescript{}{\lambda}\Delta^{'-}&:=\bar{K}_\lambda\setminus(\bar{K}_\lambda\cap-\bar{K}_\lambda),\\
\prescript{}{\lambda}\Delta^{'0}&:=(K_\lambda\cap-K_\lambda)\sqcup(\bar{K}_\lambda\cap-\bar{K}_\lambda).
\end{aligned}\]
\end{corollary}

\begin{lemma}\cite[Lemma 8]{PS}\label{PS-Lemma8}
\begin{enumerate}
\item $\Delta^{'-}\subseteq\Gamma$.
\item $\Delta^{'F}$ is a root subsystem of $\Delta''$.
\item $(\alpha,\beta)=0$, for all $\alpha\in\Delta^{'F}, \beta\in\Delta^{'I}$.
\end{enumerate}
\end{lemma}
\begin{proof}
The first statement is clear. For the second statement, from Lemma \ref{PS-Lemma5} we have $\Delta^{'F}\subseteq\Delta''$ and, by definition, $-\Delta^{'F}=\Delta^{'F}$. All we need to check is that for any $\alpha,\beta\in\Delta^{'F}, \alpha+\beta\in\Delta''$ implies $\alpha+\beta\in\Delta^{'F}$. If $\alpha+\beta\notin\Delta^{'F}$, then ${\pm(\alpha+\beta)}\cap\Gamma\neq\varnothing$. Without loss of generality, we may assume $\alpha+\beta\in\Gamma$, then $\lambda+k(\alpha+\beta)\in\supp(M)$ for all $\lambda\in\supp(M)$ and $k\in\Z_+$. Take $\lambda$ such that $\alpha\in K_\lambda$, then $(\lambda,\alpha)\geq0$. Therefore, $r_\alpha(\lambda+k(\alpha+\beta))=\lambda-(\lambda,\alpha)\alpha+k\beta\in\supp(M)$. Then Lemma \ref{PS-Lemma5} tells us that $\lambda+k\beta\in\supp(M)$ for all $k\in\Z_+$, a contradiction.

If there are $\alpha\in\Delta^{'F},\beta\in\Delta^{'I}$ such that $(\alpha,\beta)>0$, then the roots $\pm\alpha,\pm\beta,\pm r_\alpha(\beta)$ form a root system $R$ of type $A_2$. $\Gamma\cap R=\{\pm\beta\}$ or $\Gamma\cap R$ contains all roots from the half-plane of $\langle R\rangle_{\R}$ bounded by $\R\beta$. Then $\{\pm\alpha\}\cap\Gamma\neq\varnothing$, which is impossible. Thus, the third statement follows.
\end{proof}

\begin{lemma}\cite[Lemma 9]{PS}\label{PS-Lemma9} Let $\lambda\in \supp(M)$ be extremal. Then
$$K_{\lambda}=\Delta^{’+}\sqcup \{\alpha\in \Delta^{'F}|(\lambda,\alpha)\ge 0\}.$$
\end{lemma}

\begin{proof}Obviously, $(\Delta^{'I}\cup \Delta^{’-})\cap K_{\lambda}=\emptyset$. So $K_{\lambda}\subseteq\Delta^{’+}\sqcup \{\alpha\in \Delta^{’F}|(\lambda,\alpha)\ge 0\}$. We will prove that $\Delta^{’+}\sqcup \{\alpha\in \Delta^{’F}|(\lambda,\alpha)\ge 0\}\subset K_{\lambda}$.

Let us show first that $\Delta^{’+}\subset K_{\lambda}$. Let $\alpha\in \Delta^{’+}$. Choose $k\in \Z_+$ such that $\lambda+k\alpha\in \supp(M)$ but $\lambda+(k+1)\alpha\not\in \supp(M)$. From $\Delta^{’-}\subseteq \Gamma$, and Lemma \ref{PS-Lemma5} (3), $K_{\lambda}\subseteq K_{\lambda+k\alpha}$. Since $\lambda$ is extremal, $k=0$ and $\alpha\in K_{\lambda}$. It remains to show that $\{\alpha\in \Delta^{’F}|(\lambda,\alpha)\ge 0\}\subset K_{\lambda}$. Let $\alpha\in \Delta^{’F}$ with $(\lambda,\alpha)\ge 0$. Then for any $\delta\in \bar{K}_{\lambda+\alpha}$ with $\delta\ne -\alpha$, one has $(\lambda+\alpha+\delta,\alpha)=(\lambda,\alpha)+2+(\delta,\alpha)\ge (\lambda,\alpha)+1>0$. Therefore $\lambda+\alpha+\delta-\alpha=\lambda+\delta\in \supp(M)$ and $\delta\in \bar{K}_{\lambda}$. So we have proved that $K_{\lambda}\subset K_{\lambda+\alpha}$ since $-\alpha\in\bar{K}_{\lambda+\alpha}\cap\bar{K}_\lambda$.  Then by using the same argument as in the case when $\alpha\in \Delta^{’+}$, we have $\alpha\in K_{\lambda}$.
 \end{proof}

\begin{corollary}\cite[corollary 2]{PS}\label{PS-Corollary2}
If $\lambda\in\supp(M)$ is extremal, the decomposition (\ref{shadow}) is a shadow decomposition of $\Delta'$.
\end{corollary}





For an extremal $\lambda\in \supp(M)$, from Corollary \ref{PS-Corollary1} and Lemma \ref{PS-Lemma9}, 
we have
\begin{equation}\begin{split}&\prescript{}{\lambda}\Delta^{'\pm}=\Delta^{'\pm}\sqcup \{\alpha\in \Delta^{'F}|\pm (\lambda,\alpha)>0\},\\
&\prescript{}{\lambda}\Delta^{'0}=\Delta^{'I}\sqcup \{\alpha\in \Delta^{'F}|(\lambda,\alpha)=0\}.\end{split} \end{equation}
And $\Delta'=\prescript{}{\lambda}\Delta^{'+}\sqcup\prescript{}{\lambda}\Delta^{'0}\sqcup \prescript{}{\lambda}\Delta^{'-}$ is a parabolic decomposition of  $\Delta'$. Then by \cite[Lemma 2]{PS}, we get a parabolic decomposition of $\Delta$: $\Delta=\prescript{}{\lambda}\Delta^+\sqcup \prescript{}{\lambda}\Delta^0\sqcup\prescript{}{\lambda}\Delta^-$, where
\begin{equation}\begin{split}&\prescript{}{\lambda}\Delta^0:=\Delta\cap \langle\prescript{}{\lambda}\Delta^{'0}\rangle_{\Z},\\
&\prescript{}{\lambda}\Delta^{+}:=(\Delta\cap \langle\prescript{}{\lambda}\Delta^{'+}\sqcup \prescript{}{\lambda}\Delta^{'0}\rangle_{\Z_+})\setminus \prescript{}{\lambda}\Delta^0,\\
&\prescript{}{\lambda}\Delta^-:=\Delta\setminus(\prescript{}{\lambda}\Delta^0\sqcup \prescript{}{\lambda}\Delta^+).\end{split} \end{equation}

For any triangular decomposition $(\prescript{}{\lambda}\Delta^{'F,0})^+\sqcup\prescript{}{\lambda}(\Delta^{'F,0})^-$ of $\prescript{}{\lambda}\Delta^{'F,0}=\{\alpha\in \Delta^{'F}|(\lambda,\alpha)=0\}$, $\Delta'=\prescript{\lambda}{}\Delta^{'+}\sqcup\Delta^{'I}\sqcup\prescript{\lambda}{}\Delta^{'-}$ with $\prescript{\lambda}{}\Delta^{'\pm}:=\prescript{}{\lambda}\Delta^{'\pm}\sqcup(\prescript{}{\lambda}\Delta^{'F,0})^\pm$, is a parabolic decomposition. Set
\begin{equation}\label{para-decom}
\begin{aligned}
\prescript{\lambda}{}\Delta^0&:=\Delta\cap\langle\Delta^{'I}\rangle_{\Z},\\
\prescript{\lambda}{}\Delta^+&:=\Delta\cap\langle\prescript{\lambda}{}\Delta^{'+}\sqcup\Delta^{'I}\rangle_{\Z_+}\setminus\prescript{\lambda}{}\Delta^0,\\
\prescript{\lambda}{}\Delta^-&:=\Delta\setminus(\prescript{\lambda}{}\Delta^+\sqcup\prescript{\lambda}{}\Delta^0).
\end{aligned}
\end{equation}
Then (\ref{para-decom}) is also a parabolic decomposition of $\Delta$ by \cite[Lemma 2]{PS}. Denote $\prescript{\lambda}{}{\mathfrak{p}}=\h\oplus\oplus_{\alpha\in\prescript{\lambda}{}\Delta^{0}\sqcup \prescript{\lambda}{}\Delta^{+}} W_{\alpha}$.

\begin{lemma}\label{deltazero}
Let $\lambda\in\supp(M)$ be extremal.
\begin{enumerate}
\item $\langle\{\alpha\in\Delta^{'F}\,|\,(\lambda,\alpha)=0\}\rangle_\Z\cap\Delta=\{\alpha\in\Delta^{'F}\,|\,(\lambda,\alpha)=0\}$.
\item $\prescript{}{\lambda}\Delta^0\subseteq\Delta\cap\langle\Delta^{'I}\rangle_\Z+\{\alpha\in\Delta^{'F}\,|\,(\lambda,\alpha)=0\}$.
\end{enumerate}
\end{lemma}
\begin{proof}
(1) The right hand side is clearly contained in the left hand side. Clearly, we have $\Delta\cap\langle\Delta''\rangle_\Z=\Delta''$. Hence, for any $\alpha\in \langle\{\alpha\in\Delta^{'F}\,|\,(\lambda,\alpha)=0\}\rangle_\Z\cap\Delta$, we may assume $\alpha=\epsilon_i-\epsilon_j=\sum\limits_{s=1}^k\alpha_s$ with $\alpha_s\in\Delta^{'F}, (\lambda,\alpha_s)=0$. Then, there exists $s$, without loss of generality, we may assume $s=k$, such that $\alpha_k=\epsilon_i-\epsilon_l$. Therefore, $\sum\limits_{s=1}^{k-1}\alpha_s=\epsilon_l-\epsilon_j\in\Delta''$.  Then by induction on $k$ and Lemma \ref{PS-Lemma6}, we know that $\alpha\in\Delta^{'F}$.

(2) Any $\alpha\in\prescript{}{\lambda}\Delta^0$ can be written as $\alpha_1+\alpha_2$, where $\alpha_1\in\langle\Delta^{'I}\rangle_\Z, \alpha_2\in\langle\{\alpha\in\Delta^{'F}\,|\,(\lambda,\alpha)=0\}\rangle_\Z$. If $\alpha_1=0$ or $\alpha_2=0$, then by (1) the statement is clear.  Now suppose $\alpha_i\neq0, i=1,2$. Note that if $\epsilon_i-\epsilon_j\in\Delta^{'F}$, then $\pm\epsilon_i,\pm\epsilon_j,\pm(\epsilon_i-\epsilon_k),\pm(\epsilon_j-\epsilon_k)\notin\Delta^{'I}$ for all $k\neq i,j$, since $(\Delta^{'I},\Delta^{'F})=0$. So we may assume that $\alpha_2=\sum_{i,j\in C} a_{ij}(\epsilon_i-\epsilon_j),\alpha_1=\sum_{k\notin C} b_k\epsilon_k$ for some $C\subseteq\overline{1,m}$.  Then from (\ref{rootset}) we have $\alpha_1+\alpha_2=\sum_{i\neq l}m_i\epsilon_i-\epsilon_l$. Therefore, $\alpha_2=\epsilon_k-\epsilon_l$ and  $\alpha_1=\sum_{i\neq l,k} m_i\epsilon_i$ with $m_i\in\Z_+$. Therefore, $\alpha_1,\alpha_2\in\Delta$. Hence the statement follows from (1).
\end{proof}

\begin{theorem}\cite[Lemma 11]{PS}\label{extremal} One can find an extremal $\lambda\in \supp(M)$ such that $\lambda+\alpha\not\in \supp(M)$ for any $\alpha\in \prescript{\lambda}{}\Delta^{+}$. Moreover, $M$ is  the unique simple quotient of the induced module ${\rm Ind}_{U(\prescript{\lambda}{}{\mathfrak{p}})}^{U(W)} N$ for a simple weight $\prescript{\lambda}{}{\mathfrak{p}}$ module $N$ with $\supp(N)=\lambda+\langle\Delta^{'I}\rangle_{\Z}$. Furthermore, $N$ is a simple bounded module over the Levi subalgebra $\hat{\mathfrak{k}}$ of $\prescript{\lambda}{}{\mathfrak{p}}$ such that $\dim N_\mu=\dim N_\nu, \forall \mu,\nu\in\supp(N)$.
\end{theorem}

\begin{proof}
Set $S^\pm:=\{i\,|\,\pm\epsilon_i\notin\Gamma\}$. Let $\lambda\in\supp(M)$ be extremal with
\[
K^1_\lambda:=\{\beta=\epsilon_i+\epsilon_j-\epsilon_k\,|\,k\neq i,j,\lambda+\beta\notin\supp(M)\}
\]
maximal. Any $\alpha\in\prescript{}{\lambda}\Delta^+$ can be written as $\alpha=\alpha_0+\sum_{j\notin S^-}m_j\epsilon_j$, where $\alpha_0=\Delta''\cap(\prescript{}{\lambda}\Delta^{'0}\sqcup\prescript{}{\lambda}\Delta^{'+})$ or $\alpha_0=0$. Clearly, we have $\lambda+\alpha\notin\supp(M)$ whenever $\alpha_0=0$ or $\alpha_0\in\prescript{}{\lambda}\Delta^{'+}$, since $-\epsilon_j\in\Gamma, \forall j\notin S^-$. Now suppose $\alpha_0=\epsilon_j-\epsilon_k\in\Delta''\cap\prescript{}{\lambda}\Delta^{'0}$. In this case $m_i>0$ for at least one $i\in S^+$. If $\lambda+\alpha\in\supp(M)$, then $\lambda+\alpha_0+\epsilon_i,\lambda+\alpha_0\in\supp(M)$. So $\alpha_0\in\bar{K}_\lambda$ and hence $\alpha_0\in\Delta^{'I}$. Thus, $\mu=\lambda-\alpha_0\in\supp(M)$ is extremal and $\alpha_0\in\Gamma, K^1_\lambda=K^1_\mu$. However, we have $\alpha_0+\epsilon_i\in K^1_\mu\setminus K^1_\lambda$, a contradiction. Thus, we have $\lambda+\alpha\notin\supp(M), \forall \alpha\in\prescript{}{\lambda}\Delta^+$.

To finish the proof of the first statement, it remains to show $\lambda+\alpha\notin\supp(M)$ for $\alpha\in\prescript{\lambda}{}\Delta^+\setminus\prescript{}{\lambda}\Delta^+=\Delta\cap\langle(\prescript{}{\lambda}\Delta^{'F,0})^+\sqcup\Delta^{'I}\rangle_{\Z_+}\setminus\prescript{\lambda}{}\Delta^0$, which can be written as $\alpha_1+\alpha_2$ with $\alpha_1\in\langle\Delta^{'I}\rangle_\Z, \alpha_2\in(\prescript{}{\lambda}\Delta^{'F,0})^+$ from Lemma \ref{deltazero}. If $\lambda+\alpha\in\supp(M)$, then $\lambda+\alpha_2\in\supp(M)$, a contradiction to Lemma \ref{PS-Lemma9}. Thus, $\lambda+\alpha\notin\supp(M)$ for all $\alpha\in\prescript{\lambda}{}\Delta^+$. Then Lemma \ref{unique-quotient} tells us that $M$ is  the unique simple quotient of the induced module ${\rm Ind}_{U(\prescript{\lambda}{}{\mathfrak{p}})}^{U(W)} N$ for a simple weight $\prescript{\lambda}{}{\mathfrak{p}}$ module $N=U(\prescript{\lambda}{}{\mathfrak{p}})M_\lambda$. Also, from the PBW theorem, we get $\supp(N)=\lambda+\langle\Delta^{'I}\rangle_\Z$.

For the last statement, from Lemma \ref{sl2}, we know that $e_\alpha: M_\lambda\to M_{\lambda+\alpha}$ with $\alpha\in\Delta^{'I}\cap\Delta''$ or $\alpha=-\epsilon_i\in\Delta^{'I}$ is injective. Also, $e_{-\epsilon_i}: (M_*)_{-\lambda+\epsilon_i}\to (M_*)_{-\lambda}$ is injective. Here, $M_*$ is the restricted dual of $M$. Thus, the last statement follows.
\end{proof}

\begin{theorem}\label{levi}
Let $\hat{\mathfrak{k}}$ be as in Theorem \ref{extremal}. Then $\hat{\mathfrak{k}}\cong W_{q,n}\ltimes(\mathfrak{k}\otimes A_{q,n})$ for some $q\leq m$ and some Levi subalgebra $\mathfrak{k}$ of $\mathfrak{gl}_{m-q}$.
\end{theorem}

\begin{proof}
From the definition of $\prescript{\lambda}{}{\mathfrak{p}}$ and $\Delta(W_m)=\Delta$, we have $\Delta(\hat{\mathfrak{k}})=\Delta(\hat{\mathfrak{k}}\cap W_m)$. And hence $\hat{\mathfrak{k}}\cap W_m$ is isomorphic to $W_q\ltimes(\mathfrak{k}\otimes A_{q,0})$ (see \cite[Lemma 12]{PS}) for some $q$ and some Levi subalgebra $\mathfrak{k}$ of $\mathfrak{gl}_{m-q}$. Clearly, for any $\gamma\in\Delta(W_q)\cup\{0\}$, $W_{\gamma}=\h\otimes t^\gamma A_{0,n}+\sum\limits_{i=1}^n\xi_i\partial_{\xi_i}t^\gamma A_{0,n}\subseteq W_{q,n}+\h\otimes A_{q,n}$, so $\sum_{\alpha\in\Delta(W_q)\cup\{0\}}W_\alpha= W_{q,n}+\h\otimes A_{0,n}$.  Since $\mathfrak{k}$ is a Levi subalgebra of $\mathfrak{gl}_{m-q}$, any $\alpha\in\Delta(\mathfrak{k})$ is of the form $\epsilon_i-\epsilon_j$ if $q<m-1$. And  $W_{\epsilon_i-\epsilon_j+\gamma}=t_i t^\gamma A_{0,n}\partial_j$ when $q<m-1$.  If $q\ge m-1$, then $\Delta(\mathfrak{k})=\varnothing$. Thus, lemma follows.
\end{proof}

\section{Simple modules in $\mathscr{B}_{\hat{\h}}(\hat{\mathfrak{k}})$}

Let $\hat{\mathfrak{k}}=W_{q,n}\ltimes(\mathfrak{k}\otimes A_{q,n})$, where $(q,n)\neq(0,0)$ and $\mathfrak{k}$ a Levi subalgebra in $\mathfrak{gl}_{m-q}$.  Let $\hat{\h}=\h_{q}\oplus\h'$ where $\h'$ is the Cartan subalgebra of $\mathfrak{k}$. The category $\mathscr{B}_{\hat{\h}}(\hat{\mathfrak{k}})$ is defined similarly.  In this section, we will classify all simple modules $N\in\mathscr{B}_{\hat{\h}}(\hat{\mathfrak{k}})$. If $(\mathfrak{k}\otimes A_{q,n})N=0$, then $N$ is a simple module in $\mathscr{B}_{\h_q}(W_{q,n})$, which are classified in \cite{LX}.

\begin{lemma}\cite[Theorem 4.4]{LX}\label{zero}
Let $N\in\mathscr{B}_{\hat{\h}}(\hat{\mathfrak{k}})$ be simple and nontrivial. If $(\mathfrak{k}\otimes A_{q,n})N=0$, then $N$ is a simple quotient of $F(P,V)$ for some simple $P\in\mathscr{W}_{\fH_{q,n}}(\mathcal{K}_{q,n})$, some finite dimensional simple $\gl_{q,n}$ module $V$.
\end{lemma}

Now suppose $(\mathfrak{k}\otimes A_{q,n})N\neq0$, then $N=(\mathfrak{k}\otimes A_{q,n})N$.
An $A_{q,n}\hat{\mathfrak{k}}$ module is a $W_{q,n}\ltimes(\mathfrak{k}\otimes A_{q,n}\oplus A_{q,n})$ module with $A_{q,n}$ acting associatively. Consider the adjoint action of $\hat{\mathfrak{k}}$ on $\mathfrak{k}\otimes A_{q,n}$, then $(\mathfrak{k}\otimes A_{q,n})\otimes N$ is an $A_{q,n}\hat{\mathfrak{k}}$ module by
\[
f\cdot(x\otimes g\otimes v):=x\otimes fg\otimes v, \forall f,g\in A_{q,n}, x\in\mathfrak{k}, v\in N.
\]
Consider the $\hat{\mathfrak{k}}$ module epimorphism $\theta: (\mathfrak{k}\otimes A_{q,n})\otimes N\to N$ defined by $\theta(x\otimes g\otimes v):=(x\otimes g)v$. Then $X(N)=\{u\in\ker\theta\,|\,A_{q,n}\cdot u\subseteq\ker\theta\}$ is an $A\hat{\mathfrak{k}}$ submodule of $(\mathfrak{k}\otimes A_{q,n})\otimes N$. The $A_{q,n}\hat{\mathfrak{k}}$ module $\hat{N}:=\big((\mathfrak{k}\otimes A_{q,n})\otimes N\big)/X(N)$ is called the $A_{q,n}$-cover of $N$ and $\theta$ induces a $\hat{\mathfrak{k}}$ module epimorphism $\hat{\theta}: \hat{N}\to N$.

Recall from \cite[Lemma 4.2]{LX}, for $q>0$, any module $N$ in $\mathscr{B}_{\fH_{q,n}}(W_{q,n})$  with $\dim N_\lambda\leq c, \forall \lambda\in\supp(N)$,  is annihilated by $\omega_{\alpha,\beta,I,J}^{r,j,\partial,\partial'}$ for some $r\in\N$, where \begin{multline*}
\omega_{\alpha,\beta,I,J}^{r,j,\partial,\partial'}:=\sum\limits_{i=0}^r(-1)^i\binom{r}{i}t^{\alpha+(r-i)e_j}\xi_I\partial\cdot t^{\beta+ie_j}\xi_J\partial',\\
\alpha,\beta\in\Z_+^q, I,J\subseteq\overline{1,n},  j\in\overline{1,q}, \partial,\partial'\in\{\partial_i\,|\,i\in\overline{1,q}\cup\overline{m+1,m+n}\}.\end{multline*}

For any nonzero $v\in N_\lambda$, the $W_{q,n}$ submodule generated by $v$ is bounded by $c$, and is hence annihilated by $\omega_{\alpha,\beta,I,J}^{r,j,\partial,\partial'}$ for some $r$. Hence, $N$ is annihilated by $\omega_{\alpha,\beta,I,J}^{r,j,\partial,\partial'}$ for some $r$. Thus, on $N$, for any $j\in\overline{1,q}$ and $x\in\mathfrak{k}$, we have
\begin{eqnarray*}
0&=&\sum\limits_{k=0}^1\sum\limits_{i=0}^2(-1)^{i+k}\binom{2}{i}[\omega_{\alpha+ke_j,\beta+ie_j,\emptyset,\emptyset}^{r,j,\partial_j,\partial_j},x\otimes t^{\gamma+(3-i-k)e_j}\xi_I]\\
&=&\sum\limits_{k=0}^1\sum\limits_{i=0}^2\sum\limits_{l=0}^r(-1)^{i+k+l}\binom{2}{i}\binom{r}{l}[t^{\alpha+(r+k-l)e_j}\partial_j\cdot t^{\beta+(i+l)e_j}\partial_j, x\otimes t^{\gamma+(3-i-k)e_j}\xi_I]\\
&=&\sum\limits_{k=0}^1\sum\limits_{i=0}^2\sum\limits_{l=0}^r(-1)^{i+k+l}\binom{2}{i}\binom{r}{l}\Big((\gamma_j+3-i-k)x\otimes t^{\alpha+\gamma+(r-i-l+2)e_j}\xi_I\cdot t^{\beta+(i+l)e_j}\partial_j\\
&&+(\gamma_j+3-i-k) t^{\alpha+(r+k-l)e_j}\partial_j\cdot x\otimes t^{\beta+\gamma+(2+l-k)e_j}\xi_I\Big)\\
&=&\sum\limits_{i=0}^2\sum\limits_{l=0}^r(-1)^{i+l}\binom{2}{i}\binom{r}{l}x\otimes t^{\alpha+\gamma+(r-i-l+2)e_j}\xi_I\cdot t^{\beta+(i+l)e_j}\partial_j\\
&=&\sum\limits_{i=0}^{r+2}(-1)^i\binom{r+2}{i}x\otimes t^{\alpha+\gamma+(r+2-i)e_j}\xi_I\cdot t^{\beta+ie_j}\partial_j.
\end{eqnarray*}
Hence, we have
\begin{lemma}
Any module in $\mathscr{B}_{\hat{\h}}(\hat{\mathfrak{k}})$ is annihilated by $\bar{\omega}_{\alpha,\beta,I,x}^{r,j}:=\sum\limits_{i=0}^r(-1)^i\binom{r}{i}x\otimes t^{\alpha+(r-i)e_j}\xi_I\cdot t^{\beta+ie_j}\partial_j$ for all $x\in\mathfrak{k},\alpha,\beta\in\Z_+^q, j\in\overline{1,q}$ and $r\gg0$.
\end{lemma}

\begin{lemma}
Let $N\in\mathscr{B}_{\hat{\h}}(\hat{\mathfrak{k}})$ be a simple with $(\mathfrak{k}\otimes A_{q,n})N\neq0$. Then $\hat{N}\in\mathscr{B}_{\hat{\h}}(\hat{\mathfrak{k}})$.
\end{lemma}
\begin{proof}
When $q=0$, the statement is clear. Now suppose $q>0$. Since $N$ is bounded, there exists $r\in\N$ such that $\bar{\omega}_{\alpha,\beta,I,x}^{r,j}$ annihilates $N$ for all $\alpha,\beta,I,x$ and $j$. Let $S=\mathrm{span}\{t^\alpha\xi_I\,|\,0\leq\alpha_j\leq r, I\subseteq\overline{1,n}\}$. Note that $N=W_{q,0}N+N'$, where $N'=\{v\in N\,|\, \h_qv=0\}$. We claim that $(\mathfrak{k}\otimes A_{q,n})\otimes N\subseteq(\mathfrak{k}\otimes S)\otimes W_{q,0}N+\mathfrak{k}\otimes A_{q,n}\otimes N'+X(N)$.  Then lemma follows since $\dim\mathfrak{k}<\infty, \dim S<\infty, \dim N_\mu\leq c$ for some $c$ and the dimensions of  weight spaces of $\mathfrak{k}\otimes A_{q,n}\otimes N'$ smaller than or equal to $2^nc\dim\mathfrak{k}$.

Let $|\alpha|:=\sum \alpha_j$. The claim is clear for $|\alpha|\leq r$ or $\alpha_j\leq r$ for all $j$. Now suppose $\alpha_j>r$ for some $j$, then we have
\[
\sum\limits_{i=0}^r(-1)^i\binom{r}{i}x\otimes t^{\alpha-ie_j}\xi_I\otimes t^{\beta+ie_j}\partial_j v\in X(N).
\]
Hence, with $|\alpha-ie_j|<|\alpha|$ and induction on $|\alpha|$, we get $x\otimes t^\alpha\xi_I\otimes t^\beta\partial_j v=\sum\limits_{i=0}^r(-1)^i\binom{r}{i}x\otimes t^{\alpha-ie_j}\xi_I\otimes t^{\beta+ie_j}\partial_j v-\sum\limits_{i=1}^r(-1)^i\binom{r}{i}x\otimes t^{\alpha-ie_j}\xi_I\otimes t^{\beta+ie_j}\partial_j v\in \mathfrak{k}\otimes S\otimes N+X(N)$.
\end{proof}

The following theorem tells us to classify all simple modules $N\in\mathscr{B}_{\hat{\h}}(\hat{\mathfrak{k}})$ with $(\mathfrak{k}\otimes A_{q,n})N\neq0$, it suffices to classify all simple bounded $A_{q,n}\hat{\mathfrak{k}}$ modules.

\begin{theorem}\label{cover}
Any simple module $N\in\mathscr{B}_{\hat{\h}}(\hat{\mathfrak{k}})$ with $(\mathfrak{k}\otimes A_{q,n})N\neq0$ is a simple quotient of a simple bounded $A_{q,n}\hat{\mathfrak{k}}$ module.
\end{theorem}
\begin{proof}
We have a $\hat{\mathfrak{k}}$ module epimorphism $\hat{\theta}: \hat{N}\to N$. Since $\hat{N}\in\mathscr{B}_{\hat{\h}}(\gl_q\oplus\mathfrak{k})$ with $\supp(\hat{N})\subseteq\lambda+\langle\Delta(\gl_q\oplus\mathfrak{k})\rangle_\Z$, it has finite length as a $\gl_q\oplus\mathfrak{k}$ module by \cite[Lemma 3.3]{Ma}.  Then $\hat{N}$ has a composition series
\[
0=\hat{N}_0\subset\hat{N}_1\subset\cdots\subset\hat{N}_{l-1}\subset\hat{N}_l=\hat{N}
\]
with $\hat{N}_i/\hat{N}_{i-1}$ being simple $A_{q,n}\hat{\mathfrak{k}}$ modules. Then there is a smallest $i$ such that $\hat{\theta}(\hat{N}_i)\neq0$. Since $N$ is simple, we have $\hat{\theta}(\hat{N}_i)=N, \hat{\theta}(\hat{N}_{i-1})=0$. So, we get a $\hat{\mathfrak{k}}$ module epimorphism $\hat{\theta}: \hat{N}_i/\hat{N}_{i-1}\to N$.
\end{proof}

Clearly, an $A_{q,n}\hat{\mathfrak{k}}$ module is a $U=U(W_{q,n}\ltimes(\mathfrak{k}\otimes A_{q,n}\oplus A_{q,n}))=U(A_{q,n})\cdot U(\hat{\mathfrak{k}})$ module with $A_{q,n}$ acting associatively.  It is easy to see that the left ideal $\mathcal{J}$ of $U$ generated by $t^0-1, t^\alpha\xi_i\cdot t^\beta\xi_J-t^{\alpha+\beta}\xi_I\xi_J, \alpha,\beta\in\Z_+^q, I,J\subseteq\overline{1,n}$ is an ideal. And the category of $A_{q,n}\hat{\mathfrak{k}}$ modules is naturally equivalent to the category of $\overline{U}:=U/\mathcal{J}=A_{q,n}\cdot U(\hat{\mathfrak{k}})$ modules.

For $\alpha\in\Z_+^q, I\subseteq\overline{1,n},\partial\in\{\partial_i\,|\,i\in\overline{1,q}\cup\overline{m+1,m+n}\}, x\in\mathfrak{k}$, define
\[\begin{aligned}
X_{\alpha,I,\partial}&:=\sum\limits_{\substack{0\leq \beta\leq\alpha\\J\subseteq I}}(-1)^{|\beta|+|J|+\tau(J,I\setminus J)}\binom{\alpha}{\beta}t^\beta\xi_J\cdot t^{\alpha-\beta}\xi_{I\setminus J}\partial,\\
Y_{\alpha,I,x}&:=\sum\limits_{\substack{0\leq \beta\leq\alpha\\J\subseteq I}}(-1)^{|\beta|+|J|+\tau(J,I\setminus J)}\binom{\alpha}{\beta}t^\beta\xi_J\cdot x\otimes t^{\alpha-\beta}\xi_{I\setminus J},
\end{aligned}\]
where $\binom{\alpha}{\beta}:=\prod_{i=1}^q\binom{\alpha_i}{\beta_i}, |\beta|:=\sum_{i=1}^q\beta_i$. Set $\mathcal{T}=\span \{X_{\alpha,I,\partial_i}, Y_{\beta,J,x}\,|\,\alpha,\beta\in\Z_+^q, I,J\subseteq\overline{1,n},|\alpha|+|I|>0,x\in\mathfrak{k}, i\in\overline{1,q}\cup\overline{m+1,m+n}\}$, and $\nabla=\span \{\partial_i\,|\,i\in\overline{1,q}\cup\overline{m+1,m+n}\}$.

\begin{lemma}
\begin{enumerate}
\item $[\mathcal{T},A_{q,n}]=[\mathcal{T},\nabla]=0$.
\item For any $\alpha\in\Z_+^q, I\subseteq\overline{1,n},\partial\in\{\partial_i\,|\,i\in\overline{1,q}\cup\overline{m+1,m+n}\}, x\in\mathfrak{k}$, we have
\begin{eqnarray*}\begin{split}
t^\alpha\xi_I\partial&=\sum\limits_{\substack{0\leq \beta\leq\alpha\\J\subseteq I}}(-1)^{\tau(J,I\setminus J)}\binom{\alpha}{\beta}t^\beta\xi_J\cdot X_{\alpha-\beta,I\setminus J,\partial},\\
x\otimes t^\alpha\xi_I&=\sum\limits_{\substack{0\leq \beta\leq\alpha\\J\subseteq I}}(-1)^{\tau(J,I\setminus J)}\binom{\alpha}{\beta}t^\beta\xi_J\cdot Y_{\alpha-\beta,I\setminus J,x}.
\end{split}\end{eqnarray*}
\end{enumerate}
\end{lemma}
\begin{proof}
From \cite[Lemma 3.1]{LX}, we only need to check the commutators $[\partial,Y_{\alpha,I,x}], \partial\in\{\partial_i\,|\,i\in\overline{1,q}\cup\overline{m+1,m+n}\}$ and the second equation in (2). For $i\in\overline{1,q}$, we have
\begin{eqnarray*}
[\partial_i, Y_{\alpha,I,x}]&=&\sum_{\substack{0\leqslant\beta\leqslant\alpha\\J\subset I}}(-1)^{|\beta|+|J|+\tau(J,I\setminus J)}\binom{\alpha}{\beta}\beta_it^{\beta-e_i}\xi_J\cdot x\otimes t^{\alpha-\beta}\xi_{I\setminus J}\\
&&+\sum_{\substack{0\leqslant\beta\leqslant\alpha\\J\subset I}}(-1)^{|\beta|+|J|+\tau(J,I\setminus J)}\binom{\alpha}{\beta}(\alpha-\beta)_it^\beta\xi_J\cdot x\otimes t^{\alpha-\beta-e_i}\xi_{I\setminus J}\\
&=&\sum_{\substack{0\leqslant\beta\leqslant\alpha\\J\subset I}}(-1)^{|\beta|+|J|+\tau(J,I\setminus J)}\binom{\alpha}{\beta}\beta_it^{\beta-e_i}\xi_J\cdot x\otimes t^{\alpha-\beta}\xi_{I\setminus J}\\
&&-\sum_{\substack{e_i\leqslant\beta\leqslant\alpha+e_i\\J\subset I}}(-1)^{|\beta|+|J|+\tau(J,I\setminus J)}\binom{\alpha}{\beta-e_i}(\alpha-\beta+e_i)_it^{\beta-e_i}\xi_J\cdot x\otimes t^{\alpha-\beta}\xi_{I\setminus J}\\
&=&0.
\end{eqnarray*}

$[\partial_{m+j},Y_{\alpha,I,x}]=0$ is clear for $j\notin I$.  For $l_s\in I$, we have
\begin{eqnarray*}
[\partial_{m+l_s},Y_{\alpha,I,x}]&=&\sum_{\substack{0\leqslant\beta\leqslant\alpha\\J\subset I}}(-1)^{|\beta|+|J|+\tau(J,I\setminus J)}\binom{\alpha}{\beta}t^\beta\pxi{l_s}(\xi_J)\cdot x\otimes t^{\alpha-\beta}\xi_{I\setminus J}\\
&&+\sum_{\substack{0\leqslant\beta\leqslant\alpha\\J\subset I}}(-1)^{|J|}(-1)^{|\beta|+|J|+\tau(J,I\setminus J)}\binom{\alpha}{\beta}t^\beta\xi_J\cdot x\otimes t^{\alpha-\beta}\pxi{l_s}(\xi_{I\setminus J})\\
&=&\sum_{\substack{0\leqslant\beta\leqslant\alpha\\l_s\in J\subset I}}(-1)^{|\beta|+s+|J\setminus\{l_s\}|+\tau(J\setminus\{l_s\},I\setminus J)}\binom{\alpha}{\beta}t^\beta\xi_{J\setminus\{l_s\}}\cdot x\otimes t^{\alpha-\beta}\xi_{I\setminus J}\\
&&+\sum_{\substack{0\leqslant\beta\leqslant\alpha\\l_s\notin J\subset I}}(-1)^{|\beta|+s+1+|J|+\tau(J,I\setminus(J\cup\{l_s\}))}\binom{\alpha}{\beta}t^\beta\xi_J\cdot x\otimes t^{\alpha-\beta}\xi_{I\setminus(J\cup\{l_s\})}\\
&=&(-1)^sY_{\alpha,I\setminus\{l_s\},x}+(-1)^{s+1}Y_{\alpha,I\setminus\{l_s\},x}\\
&=&0.
\end{eqnarray*}

From \cite{LX}, we have $\sum_{J\subset K}(-1)^{\tau(J,I\setminus J)+|K\setminus J|+\tau(K\setminus J,I\setminus K)+\tau(J,K\setminus J)}=0$. Hence,
\begin{eqnarray*}
&&\sum_{\substack{0\leqslant\beta\leqslant\alpha\\J\subset I}}(-1)^{\tau(J,I\setminus J)}\binom{\alpha}{\beta}t^\beta\xi_J\cdot Y_{\alpha-\beta,I\setminus J,x}\\
&=&\sum_{\substack{0\leqslant\beta\leqslant\alpha\\J\subset I}}\sum_{\substack{0\leqslant\beta'\leqslant\alpha-\beta\\J'\subset I\setminus J}}(-1)^{\tau(J,I\setminus J)+|\beta'|+|J'|+\tau(J',I\setminus(J\cup J'))}\binom{\alpha}{\beta}\binom{\alpha-\beta}{\beta'}\\
&&t^\beta\xi_J\cdot t^{\beta'}\xi_{J'}\cdot x\otimes t^{\alpha-\beta-\beta'}\xi_{I\setminus(J\cup J')}\\
&=&\sum_{\substack{0\leqslant\beta\leqslant\alpha\\K\subset I}}\sum_{\substack{0\leqslant\beta'\leqslant\alpha-\beta\\J\subset K}}(-1)^{\tau(J,I\setminus J)+|\beta'|+|K\setminus J|+\tau(K\setminus J,I\setminus K)+\tau(J,K\setminus J)}\binom{\alpha}{\beta+\beta'}\binom{\beta+\beta'}{\beta}\\
&&t^{\beta+\beta'}\xi_K\cdot x\otimes t^{\alpha-\beta-\beta'}\xi_{I\setminus K}\\
&=&\sum_{\substack{K\subset I\\J\subset K}}\sum_{\substack{0\leqslant\gamma\leqslant\alpha\\0\leqslant\beta\leqslant\gamma}}(-1)^{\tau(J,I\setminus J)+|\gamma-\beta|+|K\setminus J|+\tau(K\setminus J,I\setminus K)+\tau(J,K\setminus J)}\binom{\alpha}{\gamma}\binom{\gamma}{\beta}t^\gamma\xi_K\cdot x\otimes t^{\alpha-\gamma}\xi_{I\setminus K}\\
&=&\sum_{\substack{K\subset I\\J\subset K}}(-1)^{\tau(J,I\setminus J)+|K\setminus J|+\tau(K\setminus J,I\setminus K)+\tau(J,K\setminus J)}\sum_{0\leqslant\gamma\leqslant\alpha}\binom{\alpha}{\gamma}(1-1)^{|\gamma|}(-1)^{|\gamma|}t^\gamma\xi_K\cdot x\otimes t^{\alpha-\gamma}\xi_{I\setminus K}\\
&=&x\otimes t^\alpha\xi_I.
\end{eqnarray*}
This completes the proof.
\end{proof}

\begin{lemma}\label{T}
$\mathcal{T}$ is a Lie supersubalgebra of $\overline{U}$ isomorphic to $\mathfrak{m}\nabla\ltimes(\mathfrak{k}\otimes A_{q,n})$, where $\mathfrak{m}$ is the ideal of $A_{q,n}$ generated by $\{t_i\,|\, i\in\overline{1,q}\cup\overline{m+1,m+n}\}$.
\end{lemma}
\begin{proof}
We first show that $\mathcal{T}$ is a Lie supersubalgebra of $\overline{U}$. It is easy to see that $\{X_{\alpha,I,\partial_i}, Y_{\alpha,I,x}\,|\,\alpha\in\Z_+^q,I\subseteq\overline{1,n},i\in\overline{1,q}\cup\overline{m+1,m+n}, x\in\mathfrak{k}\}$ is a generating set of the free left $A_{q,n}$ module $A\cdot\hat{\mathfrak{k}}$. Moreover, the set is $A$-linearly independent if $x$'s are taking to be a basis of $\mathfrak{k}$.

Clearly, we have $\mathcal{T}\subseteq \mathcal{T}_1:=\{y\in A_{q,n}\cdot\hat{\mathfrak{k}}\,|\,[y,A_{q,n}]=[y,\nabla]=0\}$. Any $x\in\mathcal{T}_1$ can be written as $\sum f_i\cdot y_i+y'$ with $f_i\in A_{q,n}, y_i\in\mathcal{T}, y'\in A_{q,n}\cdot\nabla$. From $[a,y]=[a,y']=0$ for all $a\in A_{q,n}$ we know that $y'=0$. Then following from $0=[\partial_i,y]=\sum[\partial_i, f_i]\cdot y_i$ we get $f_i\in\mathbb{C}$. Hence, we have $\mathcal{T}=\mathcal{T}_1$. And therefore $\mathcal{T}$ is a Lie supersubalgebra of $\overline{U}$.

Now consider the isomorphism of superspaces $\mathfrak{m}\nabla\ltimes(\mathfrak{k}\otimes A_{q,n})\to\mathcal{T}$ and the following composition of natural homomorphisms of Lie superalgebras
\[
\mathfrak{m}\nabla\ltimes(\mathfrak{k}\otimes A_{q,n})\hookrightarrow \mathfrak{m}\cdot \nabla+A_{q,n}\cdot\mathcal{T}\to(\mathfrak{m}\cdot \nabla+A_{q,n}\cdot\mathcal{T})/(\mathfrak{m}\cdot \nabla+\mathfrak{m}\cdot\mathcal{T})\to (A_{q,n}\cdot\mathcal{T})/(\mathfrak{m}\cdot\mathcal{T})\to\mathcal{T}.
\]
The homomorphism maps $t^\alpha\xi_I\partial$ to $X_{\alpha,I,\partial}$ and maps $x\otimes t^\alpha\xi_I$ to $Y_{\alpha,I,x}$.
\end{proof}

\begin{lemma}\label{assoisom}
The associative superalgebras $\overline{U}$ and $\mathcal{K}_{q,n}\otimes U(\mathcal{T})$ are isomorphic.
\end{lemma}

\begin{proof}
Define the map $\iota: \mathcal{K}_{q,n}\otimes U(\mathcal{T})\to\overline{U}$ by $\iota(x\otimes y)=x\cdot y$ for all $x\in\mathcal{K}_{q,n}, y\in U(\mathcal{T})$. $\iota$ is well defined since $\mathcal{T}$ is a Lie supersubalgebra of $\overline{U}$ and $\iota(K_{q,n})$ is supercommutative with $\iota(U(\mathcal{T}))$.

It is easy to see that the restriction of $\iota$ to $A_{q,n}\otimes\mathcal{T}+(A_{q,n}\nabla+A_{q,n})\otimes\C$ gives a Lie superalgebra isomorphism between $A_{q,n}\otimes\mathcal{T}+(A_{q,n}\nabla+A_{q,n})\otimes\C$ and $A_{q,n}\cdot \hat{\mathfrak{k}}+A_{q,n}$. Then restriction of the isomorphism from $A_{q,n}\cdot \hat{\mathfrak{k}}+A_{q,n}$ to $A_{q,n}\otimes\mathcal{T}+(A_{q,n}\nabla+A_{q,n})\otimes\C$ gives a Lie superalgebra homomorphism from $W_{q,n}\ltimes(\mathfrak{k}\otimes A_{q,n}\oplus A_{q,n})$ to $\mathcal{K}_{q,n}\otimes U(\mathcal{T})$, which induces an associative superalgebra homomorphism from $U(W_{q,n}\ltimes(\mathfrak{k}\otimes A_{q,n}\oplus A_{q,n}))$ to $\mathcal{K}_{q,n}\otimes U(\mathcal{T})$ with $\mathcal{J}$ contained in the kernel. Hence, we have the induced associative superalgebra homomorphism from $\overline{U}$ to $\mathcal{K}_{q,n}\otimes U(\mathcal{T})$, which is the inverse map of $\iota$.
\end{proof}

\begin{lemma}\label{isom}
\begin{enumerate}
\item $\big(\mathfrak{m}\nabla\ltimes(\mathfrak{k}\otimes A_{q,n})\big)/\big(\mathfrak{m}^2\nabla\ltimes(\mathfrak{k}\otimes \mathfrak{m})\big)\cong\mathfrak{gl}_{q,n}\oplus\mathfrak{k}$.
\item Any simple module $V$ in $\mathscr{B}_{\hat{\h}}(\mathfrak{m}\nabla\ltimes(\mathfrak{k}\otimes A_{q,n}))$ is annihilated by $\mathfrak{m}^2\nabla\ltimes(\mathfrak{k}\otimes \mathfrak{m})$.
\end{enumerate}
\end{lemma}
\begin{proof}
The first statement follows from $\mathfrak{m}\nabla/\mathfrak{m}^2\nabla\cong\mathfrak{gl}_{q,n}$, which holds by \cite[Lemma 3.7]{LX}.

For the second statement, since $V\in\mathscr{B}_{\hat{\h}}(\mathfrak{m}\nabla\ltimes(\mathfrak{k}\otimes A_{q,n}))$, there exists a nonzero weight vector in $V$ such that $dv=cv$ for some $c\in\C$, where $d:=\sum\limits_{i=1}^{q}d_i+\sum\limits_{i=m+1}^{m+n}d_i$. The simplicity of $V$ tells us that $V=U(\mathfrak{m}\nabla\ltimes(\mathfrak{k}\otimes A_{q,n}))v$. Because the adjoint action of $d$ on $\mathfrak{m}\nabla\ltimes(\mathfrak{k}\otimes A_{q,n})$ is diagonalizable, we know $d$ acts diagonalizbly on $V$ with eigenvalues in $c+\Z_+$. $(\mathfrak{m}^2\nabla\ltimes(\mathfrak{k}\otimes \mathfrak{m}))V$ is a submodule of $V$ with eigenvalues in $c+\N$, since $\mathfrak{m}^2\nabla\ltimes(\mathfrak{k}\otimes \mathfrak{m})$ is an ideal. Thus, $(\mathfrak{m}^2\nabla\ltimes(\mathfrak{k}\otimes \mathfrak{m}))V\neq V$, and is therefore $0$.
\end{proof}

We now have the associative superalgebra homomorphism
\begin{multline*}
\pi: \overline{U}\to\mathcal{K}_{q,n}\otimes U(\mathcal{T})\to\mathcal{K}_{q,n}\otimes U(\mathfrak{m}\nabla\ltimes(\mathfrak{k}\otimes A_{q,n}))\\
\to\mathcal{K}_{q,n}\otimes U(\big(\mathfrak{m}\nabla\ltimes(\mathfrak{k}\otimes A_{q,n})\big)/\big(\mathfrak{m}^2\nabla\ltimes(\mathfrak{k}\otimes \mathfrak{m})\big))\to\mathcal{K}_{q,n}\otimes U(\mathfrak{gl}_{q,n})\otimes U(\mathfrak{k})
\end{multline*}
with
\begin{equation*}\label{pi}\begin{split}
\pi(t^\alpha\xi_I)=&t^\alpha\xi_I\otimes1\otimes 1,\\
\pi(t^\alpha\xi_I\partial_i)=&t^\alpha\xi_I\partial_{i}\otimes 1\otimes1+\sum_{s=1}^q\partial_{s}(t^\alpha\xi_I)\otimes E_{s,i}\otimes1+(-1)^{|I|-1}\sum_{l=1}^{n}\partial_{m+l}(t^\alpha\xi_I)\otimes E_{q+l,i}\otimes1,\\
\pi(x\otimes t^\alpha\xi_I)&=t^\alpha\xi_I\otimes1\otimes x,
\end{split}\end{equation*}
where $\alpha\in\Z_+^q, I\subseteq\overline{1,n}, i\in\overline{1,q}\cup\overline{m+1,m+n}, x\in\mathfrak{k}$. Let $P$ be an $\mathcal{K}_{q,n}$ module, $M$ be a $\mathfrak{gl}_{q,n}$ module and $S$ be a $\mathfrak{k}$ module. Then we have the tensor module $\mathcal{F}(F(P,M),S):=(P\otimes M\otimes S)^\pi$ over $\hat{\mathfrak{k}}$, which is an $A_{q,n}\hat{\mathfrak{k}}$ module, with the action given by
\[
x\cdot(u\otimes v\otimes s):=\pi(x)(u\otimes v\otimes s),\, x\in\overline{U}, u\in P, v\in M, s\in S.
\]
Note that as $A_{q,n}W_{q,n}$ module, $(P\otimes M\otimes s)^\pi\cong F(P,M)$ for any $s\in S$. Since
\[\begin{aligned}
\pi(d_i)&=
d_i\otimes 1\otimes 1+1\otimes E_{i,i}\otimes 1,  d_i\in W_{q,n},\\
\pi(h)&=1\otimes 1\otimes h,  h\in\h',
\end{aligned}\]
we know that $\mathcal{F}(F(P,M),S)$ is a weight module if $P,M,S$ are weight modules over $\mathcal{K}_{q,n}, \mathfrak{gl}_{q,n}$ and $\mathfrak{k}$, respectively. Moreover, $\mathcal{F}(F(P,M),S)\in\mathscr{B}_{\hat{\h}}(\mathfrak{k})$ if and only if $F(P,M)\in\mathscr{B}_{\h_q}(W_{q,n})$ and $S\in\mathscr{B}_{\h'}(\mathfrak{k})$.

\begin{lemma}Let $P, M, S$ be simple module over $\mathcal{K}_{q,n}, \mathfrak{gl}_{q,n}$ and $\mathfrak{k}$, respectively. Then $\mathcal{F}(F(P,M),S)$ is a simple $\hat{\mathfrak{k}}$ module if and only if one of the following hold
\begin{enumerate}
\item $S$ is nontrivial;
\item $S$ is trivial and $F(P,M)$ is simple as $W_{q,n}$ module.
\end{enumerate}
\end{lemma}
\begin{proof}
When $S$ is trivial the statement is clear. So it suffices to show $\mathcal{F}(F(P,M),S)$ is simple over $\hat{\mathfrak{k}}$ when $S$ is nontrivial, which is suffices to show any $\hat{\mathfrak{k}}$ submodule $V$ of $\mathcal{F}(F(P,M),S)$ is $A_{q,n}$-invariant since $\mathcal{F}(F(P,M),S)$ is a simple $A_{q,n}\hat{\mathfrak{k}}$ module. For any nonzero $v\in V$, we may write $v=\sum\limits_{i=1}^k v_i\otimes s_i$ with $v_i\in P\otimes M, s_i\in S$ and $s_i$'s being linearly independent. From the Theorem of Density, there exists $x\in U(\mathfrak{k})$ such that $xs_i=\delta_{i,1}s_i$. Therefore, $(1\otimes 1\otimes x)v=v_1\otimes s_1\in V$. From the simplicity of $S$, we get $v_1\otimes S\subseteq V$. Then $(\mathfrak{k}\otimes A_{q,n})\cdot(v_1\otimes S)=A_{q,n}v_1\otimes S=A_{q,n}\cdot(v_1\otimes S)\subseteq V$. Similarly, we have $A_{q,n}\cdot(v_i\otimes S)\subseteq V$ for all $i$,  which means that $V$ is $A_{q,n}$-invariant.
\end{proof}

The following theorem classify all simple bounded $A_{q,n}\hat{\mathfrak{k}}$ modules.
\begin{theorem}\label{simpleAW}
Any simple bounded $A_{q,n}\hat{\mathfrak{k}}$ module is isomorphic to some $\mathcal{F}(F(P,M),S)$ for some simple weight $\mathcal{K}_{q,n}$ module $P$, some finite dimensional simple $\mathfrak{gl}_{q,n}$ module $M$ and some simple bounded $\mathfrak{k}$ module $S$.
\end{theorem}
\begin{proof}
$D=\mathrm{span}\{t_j\cdot\partial_j\,|\, j\in\overline{1,q}\cup\overline{m+1,m+n}\}+\fH_{q,n}+\h'$ is an abelian Lie super-subalgebra of $\overline{U}$ that preserves weight spaces of simple bounded $\overline{U}$ module $V$. So there is a homogeneous weight vector $v\in V$ that is a common eigenvector of $D$.  Then from $\overline{U}\xrightarrow{\sim} \mathcal{K}_{q,n}\otimes U(\mathfrak{m}\nabla\ltimes(\mathfrak{k}\otimes A_{q,n}))$,  we may regard $V$ as a module in $\mathscr{B}_{\fH_{q,n}\otimes 1+1\otimes(\fH_{q,n}+\h')}(\mathcal{K}_{q,n}\otimes U(\mathfrak{m}\nabla\ltimes(\mathfrak{k}\otimes A_{q,n})))$ module, and 
we only need to show $V$ is isomorphic to some $P\otimes M\otimes S$. For any nonzero  weight vector $v\in V$, the weight module $\mathcal{K}_{q,n}v$ has a simple submodule $P$, which is strictly simple by  Lemma \ref{Ksimple}. Hence, $V\cong P\otimes M'$ for some simple bounded $\mathfrak{m}\nabla\ltimes(\mathfrak{k}\otimes A_{q,n})$ module $M'$. Since $M'$ is bounded, it can be viewed as a simple module over $\mathfrak{gl}_{q,n}\oplus\mathfrak{k}$ by Lemma \ref{isom}. Moreover, $M'$ contains a simple submodule $S$ over $\mathfrak{k}$ because $M'$ is bounded. Therefore, $M'\cong M\otimes S$ for some simple module $M$ over $\mathfrak{gl}_{q,n}$. Thus, $V\cong P\otimes M\otimes S$. 
\end{proof}

Thus, from Lemma \ref{quotient}, Lemma \ref{zero}, Theorem \ref{cover} and Theorem \ref{simpleAW}, we have
\begin{theorem}\label{bounded-l-mod}
Any simple module in $\mathscr{B}_{\hat{\h}}(\hat{\mathfrak{k}})$ is isomorphic to one of the following:
\begin{enumerate}
\item $\mathcal{F}(F(P,M),S)$ for some nontrivial simple $S\in\mathscr{B}_{\h'}(\mathfrak{k})$, some simple $P\in\mathscr{W}_{\fH_{q,n}}(K_{q,n})$ and some simple finite dimensional $\gl_{q,n}$ module $M$;
\item a simple module with $\mathfrak{k}\otimes A_{q,n}$ acting trivially, hence a simple module in $\mathscr{B}_{\h_q}(W_{q,n})$.
\end{enumerate}
\end{theorem}

\section{Main result}

In this section, we will give the classification of simple strong Harish-Chandra modules over $W$. First, we have the following property on tensor modules.

\begin{lemma}\label{condfdim}Let $P_1$ be simple weight $\mathcal{K}_m$ module, $V_1\in \mathscr{F}_{\h}(\gl_m)$, and $V_2$ be a finite-dimensional simple $\gl_n$ module. Then $F(P_1\otimes \Lambda(n), L(V_1\otimes V_2))\in  \mathscr{F}_{\h}(W)$ if and only if $F(P_1,V_1)\in \mathscr{F}_{\h}(W_m)$ if and only if $({\Delta''}^I_{P_1}\sqcup{\Delta''}^-_{P_1})\subset({\Delta''}^F_{V_1}\sqcup {\Delta''}^-_{V_1})$. Here $\Delta^{''I}_{P_1}=\Delta^{''}\cap\Delta_{P_1}^{'I}, \Delta^{''F}_{V_1}=\Delta^{''}\cap\Delta_{V_1}^{'
    F}, \Delta^{''-}_{P_1}=\Delta^{''}\cap\Delta^{'-}_{P_1},\Delta^{''-}_{V_1}=\Delta^{''}\cap\Delta^{'-}_{V_1}$. \end{lemma}
\begin{proof}
Lemma follows from $F(P_1\otimes \Lambda(n), L(V_1\otimes V_2))_{\gamma}=\bigoplus\limits_{\alpha+\beta=\gamma}(P_1)_\alpha\otimes\Lambda(n)\otimes L(V_1\otimes V_2)_\beta, F(P_1,V_1)_{\gamma}=\bigoplus\limits_{\alpha+\beta=\gamma}(P_1)_\alpha\otimes (V_1)_\beta$ and
\[
\dim(V_1)_\beta\dim V_2\leq\dim L(V_1\otimes V_2)_\beta\leq \dim\Lambda(\mathfrak{gl}^{-1})\dim(V_1)_{\beta}\dim V_2.
\]
The last statement follows from \cite[Theorem 3.5]{GS}.
\end{proof}

\begin{theorem}\label{submod} Any simple  module in $\mathscr{F}_\h(W)$ is isomorphic to the unique simple submodule of a tensor module $F(P,M)$  that is in $\mathscr{F}_\h(W)$, where $P$ is a simple weight $\mathcal{K}$ module, and $M$ is a simple weight $\gl$ module.\end{theorem}

\begin{proof}
Let $V\in\mathscr{F}_\h(W)$ be simple. Without loss of generality, we may assume $\epsilon_i\in\Delta_V^{'I}, i\in\overline{1,q}, \epsilon_i\in\Delta_V^{'+}, i\in\overline{q+1,p+q}, \epsilon_i\in\Delta_V^{'-}, i\in\overline{p+q+1,m}$.

From Theorem \ref{extremal}, $V$ is isomorphic to the unique simple quotient of $\mathrm{Ind}_{\mathfrak{p}}^{W}N$, where $\mathfrak{p}$ is a parabolic subalgebra whose Levi subalgebra is isomorphic to $\hat{\mathfrak{k}}=W_{q,n}\ltimes(\mathfrak{k}\otimes A_{q,n})$ with $\mathfrak{k}$ a Levi subalgebra of $\gl_p\oplus\gl_{m-p-q}$ (see Theorem \ref{levi}), and $N$ is a simple weight $\hat{\mathfrak{k}}$ module satisfying

(1) $\supp(N)=\lambda+\Delta(\hat{\mathfrak{k}})$ for any $\lambda\in\supp(N)$;

(2) $\dim N_\lambda=\dim N_\mu, \forall \lambda,\mu\in\supp(N)$.

If $q=m$, then the statement follows from \cite[Theorem 4.4]{LX}. Now suppose $q<m$. Let $\mathfrak{p}'=\mathfrak{p}\cap\gl$. The Levi subalgebra of $\mathfrak{p}'$ is isomorphic to $\mathfrak{k}\oplus\gl_{q,n}$.  We need to show the unique simple quotient of $\mathrm{Ind}_{\mathfrak{p}}^{W}N$ is isomorphic to the unique simple submodule of a tensor module, where from Theorem \ref{bounded-l-mod}, $N$ is isomorphic to the unique nonzero simple submodule of $\mathcal{F}(F(P,M),S)$ with simple $P\in\mathscr{W}_{\fH_{q,n}}(\mathcal{K}_{q,n}), S\in\mathscr{B}_{\h\cap\mathfrak{k}}(\mathfrak{k})$  and $M$ a simple finite dimensional $\gl_{q,n}$ module.

Let $\mathfrak{n}$ be the nilradical of $\mathfrak{p}$. Let $U$ be the $1$-dimensional $\mathfrak{k}$ module of weight $\sum\limits_{i=q+1}^{p+q}\epsilon_i$, and $\hat{S}$ be the unique simple quotient of $\mathrm{Ind}_{\mathfrak{p}'}^{\gl}(M\otimes (S\otimes U))$. Let $\tilde{P}$ be the $\mathcal{K}$ module defined by $P\otimes A_{p,0}^\sigma\otimes A_{m-p-q,0}$. We claim that $F(\tilde{P},\hat{S})^{\mathfrak{p}-top}\cong\mathcal{F}(F(P,M),S)$, where  $F(\tilde{P},\hat{S})^{\mathfrak{p}-top}:=\bigoplus\limits_{\substack{\lambda\in\supp(F(\tilde{P},\hat{S}))\\\lambda+\alpha\notin\supp(F(\tilde{P},\hat{S})), \forall \alpha\in\Delta(\mathfrak{n})}}F(\tilde{P},\hat{S})_\lambda$ is the $\mathfrak{p}$-top of $F(\tilde{P},\hat{S})$. 
Indeed, from
\begin{equation}\label{suppP}\begin{aligned}
\supp(\tilde{P})&=\supp(P)-\sum\limits_{i=q+1}^{p+q}\epsilon_i-\sum\limits_{i=q+1}^{p+q}\Z_+\epsilon_i+\sum\limits_{i=p+q+1}^m\Z_+\epsilon_i,\\
\supp(\hat{S})&\subset\supp(S)+\sum\limits_{i=q+1}^{p+q}\epsilon_i+\supp(M)-\supp(U(\mathfrak{n}\cap\gl)),
\end{aligned}\end{equation}
we get
\[
\supp(\mathcal{F}(F(P,M),S))\subset \supp(F(\tilde{P},\hat{S}))\subset\supp(\mathcal{F}(F(P,M),S))-\supp(U(\mathfrak{n}\cap\gl)).
\]
 For any $\mu\in\supp(F(\tilde{P},\hat{S}))$, write $\mu=\alpha+\beta$ with $\alpha\in\supp(\hat{S}),\beta\in\supp(\tilde{P})$. Note that $\hat{S}$ is simple, if $\alpha\notin\supp(M\otimes (S\otimes U))$, then there exists $\alpha_1\in\supp(\mathfrak{n}\cap\gl)$ such that $\alpha+\alpha_1\in\supp(\hat{S})$. Also, from (\ref{suppP}), if $\beta\notin\supp(P)-\sum\limits_{i=q+1}^{p+q}\epsilon_k$, then there exists $\beta_1\in\{\epsilon_j,-\epsilon_k\,|\,j\in\overline{q+1,p+q}, k\in\overline{p+q+1,m}\}\subseteq\supp(\mathfrak{n}\cap\gl)$, such that $\beta+\beta_1\in\supp(\tilde{P})$.
 Thus, $\mu\in\supp(F(\tilde{P},\hat{S})^{\mathfrak{p}-top})$ 
 if and only if $\mu\in\supp(\mathcal{F}(F(P,M),S))$. Then it is straightforward to verify
 \[
 F(\tilde{P},\hat{S})=(P\otimes 1\otimes 1)\otimes(M\otimes (S\otimes U))\cong F(P,M)\otimes S.
 \]

So, we have a monomorphism $\phi$ from $N$ to $F(\tilde{P},\hat{S})^{\mathfrak{p}-top}$. Hence, we have the induced homomorphism $\tilde{\phi}$ of $W$ modules from $\mathrm{Ind}_\mathfrak{p}^WN$ to $F(\tilde{P},\hat{S})$. Note that from (\ref{suppP}), $\mathfrak{n}$ acts locally nilpotently on $F(\tilde{P},\hat{S})$, hence $R\cap F(\tilde{P},\hat{S})^{\mathfrak{p}-top}\neq 0$, where $R$ is the unique simple submodule of $F(\tilde{P},\hat{S})$. On the other hand, if $\mathcal{F}(F(P,M),S)$ is not simple, its unique simple submodule is contained in any nonzero submodule.  So, the nonzero submodule $R\cap F(\tilde{P},\hat{S})^{\mathfrak{p}-top}$ of $F(\tilde{P},\hat{S})^{\mathfrak{p}-top}$ contains the unique simple submodule $\phi(N)$ of $F(\tilde{P},\hat{S})^{\mathfrak{p}-top}$. Therefore, $\tilde{\phi}(\mathrm{Ind}_\mathfrak{p}^WN)\subseteq R$. Thus, from the simplicity of $R$,  $V\cong R$.
\end{proof}

Now we can state our main result, which follows from Lemma \ref{quotient}, Lemma \ref{condfdim} and Theorem \ref{submod}.
\begin{theorem}
Let $V\in\mathscr{F}_\h(W_{m,n})$ be simple. Then $V$ is the unique simple submodule of some tensor module $F(P,M)$ that is in $\mathscr{F}_\h(W_{m,n})$. More precisely, one of the following holds.
\begin{itemize}
\item[(i)] $V$ is isomorphic to $F(P,M)$ for a simple weight $\mathcal{K}_{m,n}$-module $P$ and a simple $M\in\mathscr{F}_\h(\gl)$, such that $({\Delta''}^I_P\sqcup{\Delta''}^-_P)\subset({\Delta''}^F_M\sqcup {\Delta''}^-_M)$ and such that  $V$ is not isomorphic to a finite-dimensional fundamental representation. Here $\Delta^{''I}_P=\Delta^{''}\cap\Delta_P^{'I}, \Delta^{''F}_M=\Delta^{''}\cap\Delta_M^{'
    F}, \Delta^{''-}_P=\Delta^{''}\cap\Delta^{'-}_P,\Delta^{''-}_M=\Delta^{''}\cap\Delta^{'-}_M$.
\item[(ii)] $V$ is isomorphic to ${\rm diff}(F(P, M))$ for a simple weight $\mathcal{K}_{m,n}$-module $P$ and a finite-dimensional fundamental representation $M$ that is not isomorphic to $\Str_{m,n}$ or $\Pi(\Str_{m,n})$.
\item[(iii)] $V$ is trivial, which is the unique simple submodule of $F(A, A[0])$ or $F(A,\Pi(A[0]))$.
\end{itemize}
\end{theorem}

\vspace{0.2cm}\noindent Y. Cai: Department of Mathematics, Soochow University, Suzhou, P. R. China.  Email: yatsai@suda.edu.cn

\noindent R.L\"u: Department of Mathematics, Soochow University, Suzhou, P. R. China.  Email: rlu@suda.edu.cn

\noindent Y. Xue.: Department of Mathematics, Soochow University, Suzhou, P. R. China.  Email: yhxue00@stu.suda.edu.cn.

\end{document}